\documentclass[review]{elsarticle}
\usepackage{lineno,hyperref}
\usepackage{amsmath,amssymb,amsthm}
\everymath{\displaystyle}

\newtheorem{theorem}{Theorem}
\theoremstyle{plain}

\newtheorem{lemma}[theorem]{Lemma}

\newtheorem{corollary}[theorem]{Corollary}
\newtheorem{remark}[theorem]{Remark}

\numberwithin{equation}{section}
\numberwithin{theorem}{section}

\newcommand{\cF}{\mathcal{F}}

\newcommand{\E}{\mathbb{E}}
\newcommand{\R}{\mathbb{R}}
\newcommand{\F}{\mathbb{F}}

\newcommand{\bP}{\mathbb{P}} % \P is used by the system

\def\Ito{It\^{o} }

\numberwithin{equation}{section}
\journal{arXiv}
\begin{document}
	\begin{frontmatter}	
		\title{Quantitative uniqueness estimates for stochastic parabolic equations on the whole Euclidean space\footnote{The first two authors are supported by the National Natural Science Foundation of China under grant 11871478, the Science Technology Foundation of Hunan Province. The last two authors is supported by the National Natural Science Foundation of China under grant 11971363, and by the Fundamental Research Funds for the Central Universities under grant 2042023kf0193.}
			\tnoteref{mytitlenote}}
		
		\author[myfirstaddress]{Yuanhang Liu}
		\ead{liuyuanhang97@163.com}
		
		\author[mysecondaryaddress]{Donghui Yang}
		\ead{donghyang@outlook.com}
		
		\medskip
		
		\author[my3address]{Xingwu Zeng \corref{mycorrespondingauthor}}
		\cortext[mycorrespondingauthor]{Corresponding author}
	\ead{xingwuzeng@whu.edu.cn}
\medskip
\author[my4address]{Can Zhang}
\ead{zhangcansx@163.com}

\address[myfirstaddress]{School of Mathematics and Statistics, Central South University, Changsha 410083, China.}
\address[mysecondaryaddress]{School of Mathematics and Statistics, Central South University, Changsha 410083, China.}
\address[my3address]{School of Mathematics and Statistics, Wuhan University, Wuhan 430072, China}
\address[my4address]{School of Mathematics and Statistics, Wuhan University, Wuhan 430072, China}
		\begin{abstract}
			In this paper, a quantitative estimate of unique continuation for the stochastic heat equation with bounded potentials on the whole Euclidean space is established. This paper generalizes the earlier results in \cite{zhang2008unique} and \cite{lu2015unique} from a bounded domain to an unbounded one. 
			The proof is based on the locally parabolic-type frequency function method. An observability estimate from measurable sets in time for the same equation is also derived.
		\end{abstract}
	\end{frontmatter}

\tableofcontents

\section{Introduction}
The study of unique continuation for solutions to deterministic partial differential equations comes from the classical Cauchy-Kovalevskaya theorem (see, e.g., \cite{zuily1983uniqueness}).  Besides in the theory of partial differential equations, it is of great significance in both Inverse Problem and Control Theory (see, for instance, \cite{lavrent_ev1986ill,LiYong,zuazua2007controllability}). The classical unique continuation property is of a qualitative nature, ensuring that the solution within a given domain can be uniquely determined by its value within a suitable subdomain. After establishing the unique continuation property, a natural question arises: Can one develop a method to recover the solution within the domain only based on the values of the solution within the subdomain? The ill-posedness of the non-characteristic Cauchy problem is widely known, indicating that a minor error in the data within the subdomain can lead to uncontrollable ramifications on the solution within the domain (see, for example, \cite{hadamard1923lectures}). Hence, the stability estimate for the solution is of importance. For an introduction to this subject, we refer the reader to \cite{lavrent_ev1986ill}.

There are rich references addressing to unique continuation not only for deterministic parabolic equations (see, e.g., \cite{escauriaza2000carleman,escauriaza2001carleman,lin1990uniqueness,phung2010quantitative,phung2014bang,poon1996unique}), but also for the stochastic counterpart in bounded domains. The
result in \cite{zhang2008unique} first showed that a solution to the stochastic parabolic equation (without boundary condition)
evolving in a bounded domain $G\subset\R^N  (N \in \mathbb{N} )$ would vanish identically $\bP$-$a.s.$, provided that it vanishes in $G_0 \times (0, T )$, $\bP$-$a.s.$, where $G_0\subseteq G$. In \cite{Lu_2012}, the author obtained an interpolation inequality for stochastic parabolic equations by Carleman estimates, which implied a conditional stability result for stochastic parabolic equations.
In \cite{li2013quantitative}, the
authors proved that a solution to the stochastic parabolic equation (with a partial homogeneous Dirichlet boundary
condition on arbitrary open subset $\Gamma_0$ of $\partial G$) evolving in $G$ vanishes $\bP$-$a.s.$, provided that its normal
derivative equals zero in $\Gamma_0\times(0, T )$, $\bP$-$a.s.$ 
In \cite{lu2015unique}, the authors established a unique continuation property for stochastic parabolic equations evolving in a domain $G \subset \mathbb{R}^N$. They demonstrated that the solution can be uniquely determined based on its values on any open subdomain of $G$ at each single point of time. Moreover, when $G$ is convex and bounded, they also provided a quantitative version of unique continuation. 
In \cite{fernandez2020hardy}, the authors proved a qualitative unique continuation at two points in time for a stochastic parabolic equation with a randomly perturbed potential. This result can be considered as a variant of Hardy's uncertainty principle for stochastic parabolic evolutions. In \cite{LuYin2020}, the authors proved a local unique continuation property for stochastic hyperbolic equations without boundary conditions to solve a local state observation problem.

More recently, in \cite{liao2023stability}, the authors established a two-ball and one-cylinder inequality based on a new Carleman estimate with both time and space boundary observation terms for the stochastic parabolic equations in a bounded domain, see \cite[Section 3]{liao2023stability} for more details. They utilized these quantitative unique continuation properties to obtain the stability estimate for the determination of the unknown time-varying boundaries.

The unique continuation estimate for deterministic partial differential equations in an unbounded domain has been also widely studied over the last decade.  In \cite{ZHANG2016389}, the author proved a unique continuation estimate for the Kolmogorov equation in the whole space by a spectral inequality and a decay inequality on the Fourier transform of the solution. In \cite{WANG2019144}, the authors proved that the unique continuation estimate for the pure heat equation in $ \mathbb{R}^n $ holds if and only if the unbounded observable set is thick set. In \cite{duan2020observability, refId0}, the authors proved a global interpolation inequality for solutions of the heat equation with
bounded potential at one point of time variable using the parabolic-type frequency function method. In \cite{WangMingZhangcan2023}, the authors proved a H\"older-type interpolation inequalities of unique continuation for fractional order parabolic equations with space-time dependent potentials on a thick set. However, to the best of our knowledge, the question of the unique continuation estimate in an unbounded domain for the stochastic counterpart is still open. 

The observability inequality for stochastic parabolic equations on a bounded domain has been extensively studied over the past decades. In the case that the observation time is the entire time interval and the observation spatial region is a nonempty open subset, we refer the reader to \cite{tang2009null} and the references therein. In those works, the proofs are almost based on the method of Carleman estimates. Alternatively, when the observation time region constitutes only a subset of positive Lebesgue measure within the time interval, and the observation spatial region is a nonempty open subset, we refer the reader to \cite{lu2011some}. In a more general context, when the observation subdomain constitutes a measurable subset of positive measure in both space and time variables, we refer the reader to \cite{yang2016observability}. There are few existing results on the observability inequality for stochastic parabolic equations in an unbounded domain.

The main contribution of this paper is that we establish the quantitative estimate of unique continuation for the stochastic heat equation with bounded and time-dependent potentials on the whole space, by using the locally parabolic-type frequency function method. More precisely, we prove a H\"older-type interpolation inequality for stochastic parabolic equations (see Theorem \ref{2.3} below), which extends a result already given in \cite[Theorem 1.6]{lu2015unique} from bounded to unbounded domains. This result seems to be discussed for the first time.  As a direct application, we obtain an observability inequality from measurable sets in time for the stochastic parabolic equation.

We remark that the parabolic-type frequency function method has been well developed in \cite[Theorem 6]{apraiz2014observability}, \cite[Lemma 5]{phung2014bang}, and \cite[Theorem 1.6]{lu2015unique} for the deterministic case. In this paper, we first employ the parabolic frequency function method to derive a locally quantitative estimate of unique continuation for the stochastic heat equation with a bounded potential, where we carefully quantify the dependence of the constant on the $L^\infty$-norm of the involved potentials. Next, by the aforementioned local result and the geometry of the observation subdomains, we obtain a globally quantitative estimate at a single time point for the solutions of the stochastic heat equation with bounded potentials. Finally, we employ the telescoping method to establish the observability inequality. 

The rest  of this paper is organized as follows. Section 2 provides the formulation of the primary problem
and states the main result Theorem \ref{2.3}. In Section 3, we introduce several auxiliary lemmas, which are instrumental in proving our main theorem. Section 4 is dedicated to the proof of Theorem \ref{2.3}, while Section 5 focuses on deriving the observability inequality, i.e., Corollary \ref{Thm1}.

\section{Problem formulation and main result}
Let $(\Omega,\mathcal{F},\mathbb{F}, \bP)$ with $\mathbb{F}\triangleq\left\{\mathcal{F}_t\right\}_{t\geq 0}$ be a complete filtered probability space on which a one dimensional standard Brownian motion $\{W(t)\}_{t\geq0}$ is defined.

Let $ T>0 $ and $H$ and $V$ be two separable Hilbert spaces with inner products $\langle\cdot,\cdot\rangle_H$, $\langle\cdot,\cdot\rangle_V$ and norms $\|\cdot\|_H$, $\|\cdot\|_V$, respectively. 
\begin{itemize}
	 \item By $L^2_{\mathcal{F}_t}(\Omega;H)$, $t \geq 0$, $p \in [1,\infty)$, we denote the space consisting of all $H$-valued,  $\mathcal{F}_t$-measurable random variables $ \xi $ such that $ \mathbb{E}\|\xi\|^2_{H}<+\infty $.
	 \item By $L^p_{\mathbb{F}}(\Omega; L^q(0,T;H))$, $p,q\in[1,\infty)$, we denote the space consisting of all $H$-valued, $\mathbb{F}$-adapted processes $X(\cdot)$ such that $ \mathbb{E} \|X(\cdot)\|_{L^q(0,T;H)}^{p}<+\infty $.
	 \item By $L^\infty_{\mathbb{F}}(0,T;V)$, we denote the space consisting of all $V$-valued, $\mathbb{F}$-adapted bounded processes.
	 \item By $L^q_{\mathbb{F}}(\Omega;C([0,T];H))$, $q\in[1,\infty)$, we denote the space consisting of all $H$-valued, $\mathbb{F}$-adapted continuous processes $X(\cdot)$ such that $ \mathbb{E}\|X(\cdot)\|^q_{C([0;T];H)}<+\infty $.
\end{itemize}
In the sequel, we  simply denote $ L^p_{\mathbb{F}}(\Omega; L^p(0,T;H))$ by $L^p_{\mathbb{F}}(0,T;H)$ with $p\in[1,\infty)$.
All the above spaces are equipped with the canonical
quasi-norms.

We consider the following stochastic heat equation with a time and space dependent potential on the whole Euclidean space
\begin{equation}\label{1.1}
	\left\{ \begin{array}{lll}
		\mathrm{d}\varphi-\Delta\varphi \mathrm{d}t=a\varphi \mathrm{d}t+ b\varphi \mathrm{d}W(t),    &\mathrm{in}\  \mathbb{R}^{N}\times (0,+\infty),\\
		\varphi(0)=\varphi_{0},  &\mathrm{in}\ \mathbb{R}^{N},
	\end{array}\right.\end{equation}
where $\varphi_{0}\in L^2_{\cF_0}(\Omega;L^2(\mathbb{R}^{N}))$, $a \in L^\infty_{\F}(0,+\infty;L^\infty(\R^N))$ and $b \in L^\infty_\F(0,+\infty;W^{1,\infty}(\R^N))$. The well-posedness of stochastic evolution equations is well-known (see e.g., \cite[Theorem 3.14]{lv2021mathematical}), and the equation (\ref{1.1}) admits a unique solution $\varphi\in L^2_{\mathbb{F}}(\Omega;C([0,T];L^2(\R^N))) \cap L^2_{\mathbb{F}}(0,T;H^1(\R^N))$.

Here and throughout this paper,  for $r>0$ and $x_{0}\in \mathbb{R}^{N}$, we use
$B_{r}(x_{0})$ to denote the closed ball centered  at $x_{0}$  and of radius $r$; and
$Q_{r}(x_{0})$ to denote the smallest cube centered at $x_0$ so that $B_{r}(x_{0})\subset Q_{r}(x_{0})$.
Let $\mathrm{int}(Q_{r}(x_{0}))$ be the interior of $Q_{r}(x_{0})$.
Write $\|a\|_{\infty}\triangleq\|a\|_{L^\infty_{\F}(0,+\infty;L^\infty(\R^N))}$ and $\|b\|_{\infty}\triangleq\|b\|_{L^\infty_\F(0,+\infty;W^{1,\infty}(\R^N))}$.
We always denote by $C(\cdot)$ a generic positive constant depending on what are enclosed in the brackets.

The main result of this paper can be stated as follows.

\begin{theorem}\label{2.3} Let $0<r<R<+\infty$ and $ T>0$.
	Assume that there is a sequence $\{x_i\}_{i\geq1}\subset\mathbb R^N$ so that
	\begin{equation*}\mathbb{R}^{N}=\bigcup_{i\geq1}Q_{R}(x_{i})
		\quad \text{with}\quad \mathrm{int}(Q_{R}(x_{i}))\bigcap \mathrm{int}(Q_{R}(x_{j}))=\emptyset\quad \text{for each}\quad i\neq j\in\mathbb N.
	\end{equation*}
	Let $$\omega:=\bigcup_{i\geq1}\omega_{i}
	\quad\text{with} \ \  \omega_{i} \ \text{being an open set and }
	\ B_{r}(x_{i})\subset \omega_{i} \subset B_{R}(x_{i})\quad\text{for each}\quad i\in\mathbb N.$$ Then there are two constants $C:= C(R)>0$ and
	$\theta:=\theta(r,R)\in (0,1)$ such that for any $\varphi_{0}\in L_{\cF_0}^{2}(\Omega;L^2 (\mathbb{R}^{N}))$,
	the corresponding solution $\varphi$ of (\ref{1.1}) satisfies
	\begin{equation}\label{3.33333}
		\begin{split}
			\E\int_{\mathbb{R}^{N}}|\varphi(x,T)|^{2}\mathrm{d}x&\leq e^{C
				\left(T^{-1} +T +T (\|a\|_{\infty}+\|b\|_{\infty}^2)+\|a\|_{\infty}^{2/3}+\|b\|_{\infty}^2+1\right)}
			\left(\E\int_{\mathbb{R}^{N}}|\varphi_{0}(x)|^{2}\mathrm{d}x\right)^{\theta}\\
			 &\times\left(\E\int_{\omega}|\varphi(x,T)|^{2}\mathrm{d}x\right)^{1-\theta}.     
		\end{split}
	\end{equation}
\end{theorem}

As an immediate application of the above theorem, an observability inequality from measurable sets in time for the solution of \eqref{1.1} can be derived.
\begin{corollary}\label{Thm1}
Let $E\subset(0,T)$ be a Lebesgue  measurable subset with a positive measure.
Under the assumptions in Theorem \ref{2.3}, there exist positive constants $C=C(r, R)$,
	and $\widetilde{C}=\widetilde{C}(r, R, E)$
	so that for any $\varphi_{0}\in L^2_{\cF_0}(\Omega;L^2(\mathbb{R}^{N}))$, the corresponding solution $\varphi$ of (\ref{1.1}) satisfies
	$$
	\E\int_{\mathbb{R}^N} |\varphi(x,T)|^2\,\mathrm dx\leq
	e^{\widetilde{C}}e^{C\left(T +T (\|a\|_{\infty}+\|b\|_{\infty}^2)+\|a\|_{\infty}^{2/3}+\|b\|^2_{\infty}+1\right)}
	\E\int_{\omega\times E}|\varphi(x,t)|^{2}\mathrm dx\mathrm dt.
	$$
\end{corollary}

\medskip

\begin{remark}
Similar results as in Theorem \ref{2.3} and Corollary \ref{Thm1} have been obtained in \cite[Theorem 1.6, Theorem 1.10]{lu2015unique} on a convex and bounded domain. In this paper, we get more sharper estimates and extend them to  the case of unbounded domains. 
\end{remark}

\section{Preliminary lemmas}\label{pre}
In this section, we give three auxiliary results that will be used later.
The first two lemmas are standard estimates for solutions of (\ref{1.1}).
For the sake of completeness we provide their detailed proofs in the Appendix.

\begin{lemma}\label{lemma-1.1}
	There is a constant $C_{1}>1$ so that for any $\varphi_{0}\in L^2_{\cF_0}(\Omega;L^2(\mathbb{R}^{N}))$,
	the solution  $\varphi$ of (\ref{1.1}) satisfies
	\begin{equation}\label{1.2}
		\begin{array}{lll}
			&&\displaystyle{\sup_{t\in[T-\tau_{1},T]}}\E\int_{B_{r}(x_{0})}\varphi^{2}(x,t)\mathrm dx
			+\E\int_{T-\tau_{1}}^{T}\int_{B_{r}(x_{0})}| \nabla\varphi(x,s)|^{2}\mathrm dx\mathrm ds\\
			\\
			&\leq& C_1 \left[(R-r)^{-2}+(\tau_{2}-\tau_{1})^{-1}+\|a\|_{\infty}+\|b\|_{\infty}^2\right]
			\displaystyle{\E\int_{T-\tau_{2}}^{T}\int_{B_{R}(x_{0})}}\varphi^{2}(x,s)\mathrm dx\mathrm ds,\\
			\\
		\end{array}
	\end{equation}
	for all $0<r<R<+\infty$, $0<\tau_{1}<\tau_{2}<T$ and $x_{0}\in \mathbb{R}^{N}$.
\end{lemma}

\begin{lemma}\label{lemma-1.2}
	There is a constant $C_{2}>0$ so that for any $\varphi_{0}\in L^2_{\cF_0}(\Omega;L^2(\mathbb{R}^{N})),$
	the solution $\varphi$ of (\ref{1.1}) satisfies
	\begin{equation}\label{1.3}
		\displaystyle{\sup_{t\in[T-\tau,T]}}\E\int_{B_{R}(x_{0})}| \nabla\varphi(x,t)|^{2}\mathrm{d}x\leq C_{2}\big(R^{-4}+\tau^{-2}+\|a\|^{2}_{\infty}+\|b\|^{4}_{\infty}\big)\E\int_{T-2\tau}^{T}\int_{B_{2R}(x_{0})}\varphi^{2}
		(x,s)\mathrm{d}x\mathrm{d}s,
	\end{equation}
	for all $0<R<+\infty,\ 0<\tau<T/2$ and $x_{0}\in \mathbb{R}^{N}.$
\end{lemma}

The following auxiliary lemma is basically motivated by \cite[Lemma 3]{phung2014bang} and \cite[Lemma 2.3]{duan2020observability}.

\begin{lemma}\label{lemma-1.3}
	Let $0<2r\leq R<+\infty$ and $\delta\in (0,1]$.
	Then there are two constants $C_{3}:= C_{3}(r,\delta)>0$ and
	$C_{4}:= C_{4}(r,\delta)>0$ so that for any  $0<\tau_{1}<\tau_{2}<T$, $x_{0}\in \mathbb{R}^{N}$,
	$\varphi_{0}\in L^2_{\cF_0}(\Omega;L^2(\mathbb{R}^{N}))\setminus\{0\}$, the quantity
	\begin{equation}\label{1.8}
		\begin{split}
			h_{0}
			=&C_{3}\Bigg[\ln(1+C_{4})+\left(1+2C_{1}(1+\frac{1}{r^2})\right)
			\left(1+\frac{1}{\tau_{2}-\tau_{1}}+\|a\|^{2/3}_{\infty}+\|b\|_{\infty}^2\right)+\frac{4C_{3}}{T}
			\\
			&+(2\|a\|_{\infty}+\|b\|_{\infty}^2)T+\ln\left(\frac{\E\int_{T-\tau_{2}}^{T}\int_{Q_{R}(x_{0})}
				\varphi^{2}(x,t)\mathrm{d}x\mathrm{d}t}{\E\int_{B_{r}(x_{0})}\varphi^{2}(x,T)\mathrm{d}x}\right)\Bigg]^{-1}
		\end{split}
	\end{equation}
	(where $\varphi$ satisfies the equation (\ref{1.1}) with $\varphi_{0}\in L^2_{\cF_0}(\Omega;L^2(\mathbb{R}^{N}))\setminus\{0\}$, and $C_{1}>1$ is the constant given by Lemma~\ref{lemma-1.1}),
	has the following two properties:
	\begin{description}
		\item[($i$)] \begin{equation}\label{1.9}
			0<\left(1+4C_3 T^{-1}+(2\|a\|_{\infty}+\|b\|_{\infty}^2)T+\|a\|^{2/3}_{\infty}+\|b\|_{\infty}^2\right)h_{0}<C_{3}.
		\end{equation}
		\item[($ii$)] There is a constant $C_{5}:= C_{5}(r,\delta)>C_{3}$ so that

		\begin{equation}\label{1.10}
			e^{(2\|a\|_{\infty}+\|b\|_{\infty}^2)T}\E\int_{T-\tau_{2}}^{T}\int_{Q_{R}(x_{0})}\varphi^{2}(x,s)\mathrm{d}x\mathrm{d}s\leq e^{1+\frac{C_{5}}{h_{0}}}\E\int_{B_{(1+\delta)r}(x_{0})}\varphi^2(x,t)\mathrm{d}x
		\end{equation}
		
	\end{description}
	for each $t\in[T-\min\{\tau_{2},h_{0}\},T]$.
\end{lemma}
\begin{proof}
	For each $r'>0$, we write  $B_{r'}:= B_{r'}(x_{0})$ and
	$Q_{r'}:= Q_{r'}(x_{0})$. Since $B_{2r}\subset Q_{R}$ and
	$$
	e^{2C_{1}\left(1+r^{-2}\right)\left[1+(\tau_{2}-\tau_{1})^{-1}+\|a\|^{2/3}_{\infty}+\|b\|_{\infty}^2\right]}
	\geq C_{1} \left[r^{-2}+(\tau_{2}-\tau_{1})^{-1}+\|a\|_{\infty}+\|b\|_{\infty}^2\right],
	$$
	by (\ref{1.2}) (where $R$ is replaced by $2r$), we have
	\begin{eqnarray*}
		&e^{2C_1\left(1+r^{-2}\right)\left[1+(\tau_2-\tau_1)^{-1}+\|a\|^{2/3}_{\infty}+\|b\|_{\infty}^2\right]}
		\displaystyle{\frac{\E\int_{T-\tau_2}^T\int_{Q_R}\varphi^2\mathrm dx\mathrm dt}{\E\int_{B_r}\varphi^2(x,T)\mathrm dx}}\\
		&\geq C_1\left[r^{-2}+(\tau_2-\tau_1)^{-1}+\|a\|_{\infty}+\|b\|_{\infty}^2\right]
		\displaystyle{\frac{\E\int_{T-\tau_2}^T\int_{B_{2r}}\varphi^{2}\mathrm dx\mathrm dt}{\E\int_{B_r}\varphi^2(x,T)\mathrm dx}}\geq 1.
	\end{eqnarray*}
	Hence, (\ref{1.9}) follows immediately from (\ref{1.8}).
	
	We now turn to the proof of (\ref{1.10}).
	Let $h>0$, $\beta(x)=|x-x_{0}|^2$ and $\eta\in C_{0}^{\infty}(B_{(1+\delta)r})$
	be such that
	$$
	0\leq\eta(\cdot)\leq 1\;\;\mbox{in}\;\;B_{(1+\delta)r}\;\;\mbox{and}\;\;
	\eta(\cdot)=1\;\;\mbox{in}\;\;B_{(1+3\delta/4)r}.
	$$
	Applying first the \Ito formula to $e^{-\beta/h}\eta^2\varphi^2$, and then integrating over  $B_{(1+\delta)r}$ and taking the expectation, we get
	\begin{equation}\label{1.99999}
		\begin{split}
			&\displaystyle{} \frac{1}{2}\frac{\mathrm{d}}{\mathrm{d}t}\E\int_{B_{(1+\delta)r}}e^{-\beta/h}(\eta\varphi)^2\mathrm{d}x
			+\E\int_{B_{(1+\delta)r}}\nabla\varphi\cdot\nabla(e^{-\beta/h}\eta^2\varphi)\mathrm{d}x\\
			=&\displaystyle{}\E\int_{B_{(1+\delta)r}}ae^{-\beta/h}(\eta\varphi)^2\mathrm{d}x
			+\frac{1}{2}\E\int_{B_{(1+\delta)r}}\eta^{2}e^{-\beta/h}b^2\varphi^2\mathrm{d}x.
		\end{split}
	\end{equation}
	Since
	$$
	\nabla(e^{-\beta/h}\eta^{2}\varphi)
	=-\frac{1}{h}e^{-\beta/h}\eta^{2}\varphi\nabla\beta+2e^{-\beta/h}\eta\varphi\nabla\eta+e^{-\beta/h}\eta^{2}\nabla\varphi,$$
	by (\ref{1.99999}), we have
	\begin{equation*}
		\begin{split}
			&\frac{1}{2}\frac{\mathrm{d}}{\mathrm{d}t}\E\int_{B_{(1+\delta)r}}e^{-\beta/h}(\eta\varphi)^{2}\mathrm{d}x
			+\E\int_{B_{(1+\delta)r}}e^{-\beta/h}|\eta\nabla\varphi|^{2}\mathrm{d}x\\
			=&\E\int_{B_{(1+\delta)r}}\frac{1}{h}e^{-\beta/h}\eta^2\varphi\nabla\beta\cdot\nabla\varphi\mathrm{d}x+\E\int_{B_{(1+\delta)r}}
			-2e^{-\beta/h}\eta\varphi\nabla\eta\cdot\nabla\varphi\mathrm{d}x\\
			&+\E\int_{B_{(1+\delta)r}}ae^{-\beta/h}(\eta\varphi)^2\mathrm{d}x+\frac{1}{2}\E\int_{B_{(1+\delta)r}}\eta^{2}e^{-\beta/h}b^2\varphi^2\mathrm{d}x\\
			\leq&\E\int_{B_{(1+\delta)r}}e^{-\beta/(2h)}|\eta\nabla\varphi|
			\left(\frac{2}{h}|x-x_{0}|e^{-\beta/(2h)}\eta|\varphi|+2|\nabla\eta|e^{-\beta/(2h)}|\varphi|\right)\mathrm{d}x\\
			&+\|a\|_{\infty}\E\int_{B_{(1+\delta)r}}e^{-\beta/h}(\eta\varphi)^{2}\mathrm{d}x+\frac{1}{2}\|b\|_{\infty}^2\E\int_{B_{(1+\delta)r}}e^{-\beta/h}
			(\eta\varphi)^{2}\mathrm{d}x.
		\end{split}
	\end{equation*}
	This, along with  Cauchy-Schwarz inequality, implies that
	\begin{equation*}
		\begin{split}
			\frac{\mathrm{d}}{\mathrm{d}t}\E\int_{B_{(1+\delta)r}}e^{-\beta/h}(\eta\varphi)^{2}\mathrm{d}x
			\leq&\left[\frac{4(1+\delta)^2 r^{2}}{h^2}+2\|a\|_{\infty}+\|b\|_{\infty}^2\right]\E\int_{B_{(1+\delta)r}}e^{-\beta/h}(\eta\varphi)^{2}\mathrm{d}x\\
			&+4\E\int_{\big\{x:(1+3\delta/4)r\leq\sqrt{\beta(x)}\leq(1+\delta)r\big\}}|\nabla\eta|^{2}e^{-\beta/h}\varphi^{2}\mathrm{d}x,
		\end{split}
	\end{equation*}
	which indicates that
	\begin{equation*}
		\begin{split}
			\frac{\mathrm{d}}{\mathrm{d}t}\E\int_{B_{(1+\delta)r}}e^{-\beta/h}(\eta\varphi)^{2}\mathrm{d}x
			\leq&\left[\frac{4(1+\delta)^2 r^{2}}{h^2}+2\|a\|_{\infty}+\|b\|_{\infty}^2\right]\E\int_{B_{(1+\delta)r}}e^{-\beta/h}(\eta\varphi)^{2}\mathrm{d}x\\
			&+4\|\nabla\eta\|^{2}_{\infty}e^{-\frac{(1+3\delta/4)^2 r^2}{h}}\E\int_{B_{(1+\delta)r}}\varphi^{2}\mathrm{d}x.
		\end{split}
	\end{equation*}
	Here and throughout the proof of Lemma~\ref{lemma-1.3},
	$\|\nabla\eta\|_{\infty}:=\|\nabla\eta\|_{L^{\infty}(B_{(1+\delta)r})}$.
	From the latter it follows that
	\begin{equation*}
		\begin{split}
			&\frac{\mathrm{d}}{\mathrm{d}t}\left[e^{-\left(\frac{4(1+\delta)^2 r^2}{h^2}+2\|a\|_{\infty}+\|b\|_{\infty}^2\right)t}
			\E\int_{B_{(1+\delta)r}}e^{-\beta/h}|\eta\varphi|^{2}\mathrm{d}x\right]\\
			\leq&4\|\nabla\eta\|^{2}_{\infty}e^{-\left(\frac{4(1+\delta)^2 r^2}{h^2}+2\|a\|_{\infty}+\|b\|_{\infty}^2\right)t}
			e^{-\frac{(1+3\delta/4)^2 r^2}{h}}\E\int_{B_{(1+\delta)r}}\varphi^{2}\mathrm{d}x.
		\end{split}
	\end{equation*}
	Integrating the above inequality over $(t,T)$, we get
	\begin{equation}\label{1.109999}
		\begin{split}
			&\displaystyle{\E\int_{B_{(1+\delta)r}}}e^{-\beta/h}|\eta\varphi(x,T)|^{2}\mathrm dx\\
			\leq& e^{\left(\frac{4(1+\delta)^2 r^2}{h^2}+2\|a\|_{\infty}+\|b\|_{\infty}^2\right)(T-t)}
			\displaystyle{\E\int_{B_{(1+\delta)r}}}e^{-\beta/h}|\eta\varphi(x,t)|^{2}\mathrm{d}x\\
			&+4e^{\left(\frac{4(1+\delta)^2 r^2}{h^2}+2\|a\|_{\infty}+\|b\|_{\infty}^2\right)(T-t)}\|\nabla\eta\|^2_{\infty}
			e^{-\frac{(1+3\delta/4)^2 r^2}{h}}\displaystyle{\int_t^T\int_{B_{(1+\delta)r}}}\varphi^{2}(x,s)\mathrm{d}x\mathrm{d}s.
		\end{split}
	\end{equation}
	We simply write  $b_{1}:= 4(1+\delta)^{2}, b_{2}:= (1+3\delta/4)^{2}$
	and $b_{3}:= (1+\delta/2)^{2}.$ It is clear that $1<b_{3}<b_{2}<b_{1}$.
	Recall that $t\leq T$. We now suppose  $h>0$ to be such that
	$$
	0<T-\frac{(b_{2}-b_{3})h}{b_{1}}\leq t.
	$$
	Then $b_{1}(T-t)/h^{2}\leq(b_{2}-b_{3})/h$ and (\ref{1.109999}) yields
	\begin{equation*}
		\begin{split}
			\E\int_{B_{(1+\delta)r}}e^{-\beta/h}|\eta\varphi(x,T)|^{2}\mathrm{d}x
			\leq& e^{\frac{(b_{2}-b_{3})r^{2}}{h}}e^{(2\|a\|_{\infty}+\|b\|_{\infty}^2)T}
			\E\int_{B_{(1+\delta)r}}e^{-\beta/h}|\eta\varphi(x,t)|^{2}\mathrm{d}x\\
			&+4\|\nabla\eta\|^{2}_{\infty}e^{(2\|a\|_{\infty}+\|b\|_{\infty}^2)T}e^{\frac{-b_{3}r^{2}}{h}}
			\E\int_{t}^{T}\int_{B_{(1+\delta)r}}\varphi^{2}(x,s)\mathrm{d}x\mathrm{d}s.
		\end{split}
	\end{equation*}
	Since $\eta(\cdot)=1$ in $B_{r}$, the following estimate holds
	\begin{equation}\label{1.15}
		\begin{split}
			\displaystyle{}\E\int_{B_{r}}|\varphi(x,T)|^{2}\mathrm{d}x\leq& e^{\frac{(b_{2}-b_{3}+1)r^{2}}{h}}e^{(2\|a\|_{\infty}+\|b\|_{\infty}^2)T}
			\E\int_{B_{(1+\delta)r}}e^{-\beta/h}|\eta\varphi(x,t)|^{2}\mathrm{d}x\\
			&+4\|\nabla\eta\|^{2}_{\infty}e^{(2\|a\|_{\infty}+\|b\|_{\infty}^2)T}e^{\frac{-(b_{3}-1)r^{2}}{h}}
			\E\int_{t}^{T}\int_{B_{(1+\delta)r}}\varphi^{2}(x,s)\mathrm{d}x\mathrm{d}s,
		\end{split}
	\end{equation}
	whenever $0<T-(b_{2}-b_{3})h/b_{1}\leq t\leq T$. Recall that $h_{0}<T$ from \eqref{1.9}. We choose $h$ as follows:
	\begin{equation*}
		\begin{split}
			h=&\frac{b_{1}}{b_{2}-b_{3}}h_{0}\\
			=&\frac{b_{1}C_{3}/(b_{2}-b_{3})}
			{\ln\left[(1+C_{4})\left(e^{\left[1+2C_{1}(1+\frac{1}{r^2})\right]
					(1+\frac{1}{\tau_{2}-\tau_{1}}+\|a\|^{2/3}_{\infty}+\|b\|_{\infty}^2)+\frac{4C_{3}}{T}+(2\|a\|_{\infty}+\|b\|_{\infty}^2)T}\right)
				\frac{\E\int_{T-\tau_{2}}^{T}\int_{Q_{R}}\varphi^{2}\mathrm{d}x\mathrm{d}t}
				{\E\int_{B_{r}}\varphi^{2}(x,T)\mathrm{d}x}\right]}   
		\end{split}
	\end{equation*}
	with $C_{3}:=(b_{2}-b_{3})(b_{3}-1)r^{2}/b_{1}$ and
	$C_{4}:=4\|\nabla\eta\|^{2}_{\infty}$.
	Then for any $t\in[T-\min\{\tau_{2},h_{0}\},T]$, we have
	\begin{equation}\label{2.211111}
		\begin{split}
			&\displaystyle{}4\|\nabla\eta\|^{2}_{\infty}e^{(2\|a\|_{\infty}+\|b\|_{\infty}^2)T}
			e^{-\frac{(b_{3}-1)r^{2}}{h}}\int_{t}^{T}\E\int_{B_{(1+\delta)r}}\varphi^{2}(x,s)\mathrm{d}x\mathrm{d}s\\
			=&\displaystyle{}\frac{C_{4}e^{(2\|a\|_{\infty}+\|b\|_{\infty}^2)T}\E\int_{t}^{T}\int_{B_{(1+\delta)r}}\varphi^{2}(x,s)\mathrm{d}x\mathrm{d}s}
			{(1+C_{4})\left(e^{\left[1+2C_{1}(1+\frac{1}{r^2})\right]
					(1+\frac{1}{\tau_{2}-\tau_{1}}+\|a\|^{2/3}_{\infty}+\|b\|_{\infty}^2)+\frac{4C_{3}}{T}+(2\|a\|_{\infty}+\|b\|_{\infty}^2)T}\right)
				\frac{\E\int_{T-\tau_{2}}^{T}\int_{Q_{R}}\varphi^{2}(x,s)\mathrm{d}x\mathrm{d}s}{\E\int_{B_{r}}\varphi^{2}(x,T)\mathrm{d}x}}\\
			\leq&\displaystyle{} {\frac{1}{e}}\E\int_{B_{r}}\varphi^{2}(x,T)\mathrm{d}x.
		\end{split}
	\end{equation}
	The last inequality is implied by the facts that $(1+\delta)r\leq2r\leq R\ \mathrm{ and} \ B_{(1+\delta)r}\subset Q_{R}.$
	
	On one hand, by \eqref{1.15} and \eqref{2.211111}, we get
	\begin{equation}\label{2.22222}
		\left(1-\frac{1}{e}\right)\E\int_{B_{r}}\varphi^{2}(x,T)\mathrm{d}x\leq
		e^{\frac{(b_{2}-b_{3}+1)(b_{2}-b_{3})r^{2}}{b_{1}h_{0}}}
		e^{(2\|a\|_{\infty}+\|b\|_{\infty}^2)T}\E\int_{B_{(1+\delta)r}}|\varphi(x,t)|^{2}\mathrm{d}x
	\end{equation}
	for each $T-\min{\{\tau_{2},h_{0}\}}\leq t\leq T.$
	On the other hand, by \eqref{1.8}, we see
	\begin{equation*}
		\frac{\E\int_{T-\tau_{2}}^{T}\int_{Q_{R}}\varphi^{2}(x,s)\mathrm{d}x\mathrm{d}s}
		{\E\int_{B_{r}}\varphi^{2}(x,T)\mathrm{d}x}\leq e^{\frac{C_{3}}{h_{0}}},
	\end{equation*}
	which, combined with \eqref{2.22222}, indicates that
	$$
	\left(1-\frac{1}{e}\right)e^{-\frac{C_{3}}{h_{0}}}
	\E\int_{T-\tau_{2}}^{T}\int_{Q_{R}}\varphi^{2}(x,s)\mathrm{d}x\mathrm{d}s\leq
	e^{\frac{(b_{2}-b_{3}+1)(b_{2}-b_{3})r^{2}}{b_{1}h_{0}}}
	e^{(2\|a\|_{\infty}+\|b\|_{\infty}^2)T}\E\int_{B_{(1+\delta)r}}|\varphi(x,t)|^{2}\mathrm{d}x
	$$
	for each $T-\min{\{\tau_{2},h_{0}\}}\leq t\leq T.$
	Since $(2\|a\|_{\infty}+\|b\|_{\infty}^2)Th_{0}<C_{3}$ (see \eqref{1.9}), the desired estimate \eqref{1.10}
	follows from the latter inequality immediately with $C_{5}:=3C_{3}+(b_{2}-b_{3}+1)(b_{2}-b_{3})r^{2}/b_{1}$.
\end{proof}

\section{Proof of Theorem \ref{2.3}}
In this section, we shall study the quantitative version of unique continuation for the solution of \eqref{1.1}, i.e., Theorem \ref{2.3}. In what follows, for each $\lambda>0$, and $x_0\in\R^N$, we define 
\begin{equation}\label{G}
	G_{\lambda}(x,t)\triangleq\frac{1}{(T-t+\lambda)^{N/2}}e^{-\frac{|x-x_{0}|^{2}}{4(T-t+\lambda)}},
	\;\;t\in [0,T],\;\;x\in \mathbb{R}^N .
\end{equation}
It is clear that
\begin{equation}\label{G-some}
	\left\{ \begin{aligned}
		&\partial_tG_{\lambda}(x,t)+\Delta G_{\lambda}(x,t)=0,\quad\ \nabla G_{\lambda}(x,t)=-\frac{x-x_0}{2(T-t+\lambda)}G_{\lambda}(x,t),\\
		&\Delta G_{\lambda}(x,t)=-\frac{N}{2(T-t+\lambda)}G_{\lambda}(x,t)+\frac{|x-x_0|^2}{4(T-t+\lambda)^2}G_{\lambda}(x,t),\\
		&\partial_{x_ix_j}G_{\lambda}(x,t)=\frac{(x_i-x_{0i})(x_j-x_{0j})}{4(T-t+\lambda)^2}G_{\lambda}(x,t), \quad i\neq j.
	\end{aligned}\right.
\end{equation}

For $\delta\in(0,1]$, $R>0$, we denote $R_0:=(1+2\delta)R$. Let $\chi\in C_{0}^{\infty}(B_{R_{0}})$ be such that
\begin{equation}\label{chi}
	0\leq\chi(\cdot)\leq 1 \ \mathrm{in} \ B_{R_{0}} \ \mathrm{and} \ \chi(\cdot)=1 \ \mathrm{in} \ B_{(1+3\delta/2)R}.
\end{equation}

We set 
\begin{equation}\label{u-F}
	u:=\chi\varphi,\quad\ F:=au-\varphi\Delta\chi-2\nabla\varphi\cdot\nabla\chi.
\end{equation} 
Then one can verify that
\begin{equation}\label{3.222}
	\mathrm{d}u-\Delta u\mathrm{d}t=F\mathrm{d}t+bu\mathrm{d}W(t) \  \;\mathrm{in} \;\ B_{R_{0}}\times(0,T).
\end{equation}
Denote
\begin{equation}\label{HDN}
	\left\{ \begin{aligned}
		&H_{\lambda,R_0}(t)=\E\int_{B_{R_0}(x_{0})}|u(x,t)|^{2}G_{\lambda}(x,t)\mathrm{d}x,\\
		&D_{\lambda,R_0}(t)=\E\int_{B_{R_0}(x_{0})}|\nabla u(x,t)|^{2}G_{\lambda}(x,t)\mathrm{d}x,\\
		&N_{\lambda,R_0}(t)=\frac{2D_{\lambda,R_0}(t)}{H_{\lambda,R_0}(t)},~\text{whenever}~ H_{\lambda,R_0}(t)\neq0 .
	\end{aligned}\right.
\end{equation}
Throughout this section, we always work under the assumption $H_{\lambda,R_0}(\cdot)\neq0$.

\begin{lemma}\label{lemma-dH}
	For the function $H_{\lambda,R_0}(\cdot)$ defined in \eqref{HDN}, involving the solution $\varphi$ to the equation \eqref{1.1} over the ball $B_{R_0}(x_0)$, it holds that
	\begin{equation}\label{eq-dH}
		\frac{\mathrm{d}}{\mathrm{d}t}H_{\lambda,R_0}(t)=-2D_{\lambda,R_0}(t)+2\E\int_{B_{R_0}(x_{0})}uFG_{\lambda}(x,t)\mathrm{d}x+\E\int_{B_{R_0}(x_{0})}b^2u^2G_{\lambda}(x,t)\mathrm{d}x.
	\end{equation}
\end{lemma}

For simplicity, we denote
$$
\|b\|^2_{L^\infty_{\F}(0,+\infty;W^{1,\infty}(B_{R_0}(x_0)))}:=\|b\|^2_{B_{R_0}}.
$$
Next, we introduce the following monotonicity of the parabolic-type frequency function
associated with stochastic parabolic equations.
\begin{lemma}\label{lemma-2.1}
	For the function $N_{\lambda,R_0}(\cdot)$ defined in \eqref{HDN}, involving the solution $\varphi$ to the equation \eqref{1.1} over the ball $B_{R_0}(x_0)$, it follows that
	\begin{equation}\label{lemma-2.1-monotonicity}
		\frac{\mathrm{d}}{\mathrm{d}t}N_{\lambda,R_0}(t)\leq\left(\frac{1}{T-t+\lambda}+2\|b\|^2_{B_{R_0}(x_0)}\right)N_{\lambda,R_0}(t)+2\|b\|^2_{B_{R_0}(x_0)}
		+\frac{\E\int_{B_{R_0}(x_{0})}F^2G_{\lambda}(x,t)\mathrm{d}x}{H_{\lambda,R_0}(t)}.
	\end{equation}
\end{lemma}

\begin{remark}
	Lemma \ref{lemma-dH} and Lemma \ref{lemma-2.1} were proved in \cite[Lemma 2.1]{lu2015unique} and \cite[Lemma 2.2]{lu2015unique} for a bounded and convex domain. By a similar argument, the same results can be obtained. Hence, we omit the detailed proofs here.
\end{remark}

We then have the following two-ball and one-cylinder inequality, which is inspired by \cite[Theorem 2]{escauriaza2006doubling} and \cite[Lemma 3.2]{duan2020observability}. Its proof here is adapted from  \cite[Lemma 4]{phung2014bang}
by using Lemma \ref{lemma-1.3} instead.
\begin{lemma}\label{lemma-2.2}
	Let $0<r<R<+\infty$ and  $\delta\in(0,1]$.
	Then there are three positive constants $C_{6}:= C_{6}(R,\delta), C_{7}:= C_{7}(R,\delta)$ and $\gamma:=\gamma(r,R,\delta)\in (0,1)$ so that for any $x_{0}\in \mathbb{R}^{N}$ and any $\varphi_{0}\in L^2_{\cF_0}(\Omega;L^2(\mathbb{R}^{N}))$,  the solution $\varphi$ of (\ref{1.1}) satisfies
	\begin{equation*}
		\begin{split}
			&\E\int_{B_{R}(x_{0})}|\varphi(x,T)|^{2}\mathrm{d}x\\
			\leq&\left[C_{6}e^{[1+2C_{1}(1+\frac{1}{R^2})](1+\frac{4}{T}+\|a\|^{2/3}_{\infty}+\|b\|^{2}_{\infty})
				+\frac{C_{7}}{T}+(2\|a\|_{\infty}+\|b\|^{2}_{\infty})T}\E\int_{T/2}^{T}\int_{Q_{2R_{0}}(x_{0})}\varphi^{2}(x,t)\mathrm{d}x\mathrm{d}t\right]^{\gamma}\\
			&\times\left(2\E\int_{B_{r}(x_0)}|\varphi(x,T)|^{2}\mathrm{d}x\right)^{1-\gamma},
		\end{split}
	\end{equation*}
	where $C_{1}$ is the constant given by Lemma~\ref{lemma-1.1}.
\end{lemma}

\begin{remark}
	A similar result is obtained in \cite[Theorem 3.1]{liao2023stability} for the stochastic parabolic equation on a time-varying domain. Their proof is based on the Carleman estimate, while ours is based on the parabolic-type frequency function and quantify the dependence of the constant on the $L^\infty$-norm of the involved potentials.
\end{remark}

\begin{proof}[Proof of Lemma~\ref{lemma-2.2}]
	For each $r'>0$, we denote $B_{r'}:= B_{r'}(x_{0})$ and
	$Q_{r'}:= Q_{r'}(x_{0})$.  
	Furthermore, we define  
	\begin{equation}\label{g}
		g:=-2\nabla\chi\cdot\nabla\varphi-\varphi\Delta\chi.
	\end{equation}

	\textbf{Step 1}. Note that $g$ is supported on $\{x: (1+3\delta/2)R\leq|x-x_{0}|\leq R_{0}\}.$
	Recall that $\chi(\cdot)=1$ in $B_{(1+\delta)R}$ (see \eqref{chi}). We can easily check that
	\begin{equation}\label{3.333}
		\begin{split}
			\displaystyle{}&\frac{\E\int_{B_{R_{0}}} u(x,t)g(x,t)G_{\lambda}(x,t)\mathrm{d}x}
			{H(t)}\\
			=&\displaystyle{}\frac{\E\int_{B_{R_{0}}\setminus B_{(1+3\delta/2)R}}
				\chi\varphi(-2\nabla\chi\cdot\nabla\varphi-\varphi\Delta\chi)
				e^{-\frac{|x-x_{0}|^{2}}{4(T-t+\lambda)}}\mathrm{d}x}
			{\E\int_{B_{R_{0}}}|\chi\varphi(x,t)|^{2}e^{-\frac{|x-x_{0}|^{2}}{4(T-t+\lambda)}}\mathrm{d}x}\\
			\leq& \displaystyle{} e^{-\frac{\mathcal{K}_{1}}{T-t+\lambda}}\frac{\E\int_{B_{R_{0}}\setminus B_{(1+3\delta/2)R}}\left(2|\varphi\nabla\chi\cdot\nabla\varphi|
				+|\Delta\chi|\varphi^{2}\right)\mathrm{d}x}{\E\int_{B_{(1+\delta)R}}\varphi^{2}(x,t)\mathrm{d}x}\\
			\leq& \displaystyle{} e^{-\frac{\mathcal{K}_{1}}{T-t+\lambda}}\frac{2\|\nabla\chi\|_{\infty}
				(\E\int_{B_{R_{0}}}\varphi^{2}(x,t)\mathrm{d}x)^{\frac{1}{2}}(\E\int_{B_{R_{0}}}
				|\nabla\varphi(x,t)|^{2}\mathrm{d}x)^{\frac{1}{2}}+\|\Delta\chi\|_{\infty}\E\int_{B_{R_{0}}}\varphi^{2}(x,t)\mathrm{d}x}
			{\E\int_{B_{(1+\delta)R}}\varphi^{2}(x,t)\mathrm{d}x},
		\end{split}
	\end{equation}
	where $\mathcal{K}_{1}:=[(1+3\delta/2)R]^2/4-[(1+\delta)R]^2/4$ and $\|\nabla\chi\|_{\infty}:=\|\nabla\chi\|_{L^{\infty}(B_{R_{0}})}$ and $\|\Delta\chi\|_{\infty}:=\|\Delta\chi\|_{L^{\infty}(B_{R_{0}})}.$

	On one hand, by Lemma~\ref{lemma-1.1} (where $r, R,\tau_{1} $ and $\tau_{2}$ are
	replaced by
	$R_{0}, 2R_{0}, T/4$ and $T/2$, respectively), we have
	\begin{equation}\label{3.555}
		\E\int_{B_{R_{0}}}\varphi^{2}(x,t)\mathrm{d}x
		\leq \mathcal{K}_{2}(1+T^{-1}+\|a\|_{\infty}+\|b\|_{\infty}^2)\E\int_{T/2}^{T}\int_{B_{2R_{0}}}\varphi^{2}(x,t)
		\mathrm{d}x\mathrm{d}t \  \ \mathrm{for \ each} \ t\in [3T/4,T],
	\end{equation}
	where $\mathcal{K}_{2}:=\mathcal{K}_{2}(R)>0.$
	By Lemma~\ref{lemma-1.2} (where $ R $ and $\tau$ are replaced by  $R_{0}$ and $T/4$, respectively),
	we get that for each $t\in [3T/4,T],$
	\begin{equation}\label{3.666}
		\E\int_{B_{R_{0}}}|\nabla\varphi(x,t)|^{2}\mathrm{d}x\leq \mathcal{K}_{3}(1+T^{-2}+\|a\|^{2}_{\infty}+\|b\|^{4}_{\infty})\E\int_{T/2}^{T}\int_{B_{2R_{0}}}\varphi^{2}(x,s)\mathrm{d}x\mathrm{d}s,
	\end{equation}
	where $\mathcal{K}_{3}:=\mathcal{K}_{3}(R)>0.$
	By \eqref{1.10} in Lemma \ref{lemma-1.3} (where $r, R,\tau_{1} $ and $\tau_{2}$ are replaced by
	$R, 2R_{0}, T/4$ and $T/2$, respectively), it holds that
	\begin{equation}\label{3.888}
		\begin{split}
			\displaystyle{}e^{(2\|a\|_{\infty}+\|b\|_{\infty}^2)T}\E\int_{T/2}^{T}\int_{B_{2R_{0}}}\varphi^{2}\mathrm{d}x\mathrm{d}s\leq& \displaystyle{}e^{(2\|a\|_{\infty}+\|b\|_{\infty}^2)T}\E\int_{T/2}^{T}\int_{Q_{2R_{0}}}\varphi^{2}\mathrm{d}x\mathrm{d}s\\
			\leq& \displaystyle{} e^{1+\frac{C_{5}}{h_{0}}}
			\E\int_{B_{(1+\delta)R}}\varphi^{2}(x,t)\mathrm{d}x \  \ \mathrm{for \ each} \ t\in [T-h_{0},T].
		\end{split}
	\end{equation}
	Here, we used the fact that $h_{0}<T/4$ (see \eqref{1.9} in Lemma \ref{lemma-1.3}).
	It follows from \eqref{3.333}-\eqref{3.888} that
	\begin{equation}\label{3.777}
		\begin{split}
			&\displaystyle{}\frac{\E\int_{B_{R_{0}}} u(x,t)g(x,t)G_{\lambda}(x,t)\mathrm{d}x}
			{H(t)}\\
			\leq& \displaystyle{} e^{-\frac{\mathcal{K}_{1}}{T-t+\lambda}}
			\frac{\mathcal{K}_{4}(1+T^{-2}+\|a\|^{3/2}_{\infty}+\|b\|^{4}_{\infty})
				\E\int_{T/2}^{T}\int_{B_{2R_{0}}}\varphi^{2}(x,s)\mathrm{d}x\mathrm{d}s}
			{\E\int_{B_{(1+\delta)R}}\varphi^{2}(x,t)\mathrm{d}x}\\
			\leq& \displaystyle{} \mathcal{K}_{4}e^{-\frac{\mathcal{K}_{1}}{T-t+\lambda}}e^{1+\frac{C_{5}}{h_{0}}}(1+T^{-2})
			\ \ \  \mathrm{for \ each}  \ \ t\in [T-h_{0},T].
		\end{split}
	\end{equation}
	where $\mathcal{K}_{4}:=\mathcal{K}_{4}(R, \delta)>0.$

	On the other hand, by similar arguments as those for \eqref{3.777}, we have
	\begin{equation}\label{3.111111}
		\begin{split}
			\int_{t}^{T}\frac{\E\int_{B_{R_{0}}} |g(x,s)|^{2}G_{\lambda}(x,s)\mathrm{d}x}{H(s)}\mathrm{d}s
			\leq& \int_{t}^{T}\frac{\E\int_{B_{R_{0}}} |-2\nabla\chi\cdot\nabla\varphi-\varphi\Delta\chi|^{2}\mathrm{d}x}
			{\E\int_{B_{(1+\delta)R}}|\varphi(x,s)|^{2}\mathrm{d}x}e^{-\frac{\mathcal{K}_{1}}{T-s+\lambda}}\mathrm{d}s\\
			\leq&\int_{t}^{T}\frac{8\|\nabla\chi\|^{2}_{\infty}\E\int_{B_{R_{0}}} |\nabla\varphi|^{2}\mathrm{d}x+2\|\Delta\chi\|^{2}_{\infty}
				\E\int_{B_{R_{0}}}\varphi^{2}\mathrm{d}x}{\E\int_{B_{(1+\delta)R}}|\varphi(x,s)|^{2}\mathrm{d}x}
			e^{-\frac{\mathcal{K}_{1}}{T-s+\lambda}}\mathrm{d}s\\
			\leq&\displaystyle{}\mathcal{K}_{5}(1+T^{-2}+\|a\|^{2}_{\infty}+\|b\|^{4}_{\infty})
			\int_{t}^{T}\frac{\E\int_{T/2}^{T}\int_{B_{2R_{0}}}\varphi^{2}\mathrm{d}x\mathrm{d}s}
			{\E\int_{B_{(1+\delta)R}}|\varphi(x,s)|^{2}\mathrm{d}x}e^{-\frac{\mathcal{K}_{1}}{T-s+\lambda}}\mathrm{d}s\\
			\leq&\displaystyle{}\mathcal{K}_{5}(1+T^{-2}+\|a\|^{2}_{\infty}+\|b\|^{4}_{\infty})
			e^{1+\frac{C_{5}}{h_{0}}}e^{-(2\|a\|_{\infty}+\|b\|^{2}_{\infty})T}\int_{t}^{T}e^{-\frac{\mathcal{K}_{1}}{T-s+\lambda}}\mathrm{d}s\\
			\leq&\displaystyle{}\mathcal{K}_{5}(1+T^{-2})
			e^{1+\frac{C_{5}}{h_{0}}}e^{-\frac{\mathcal{K}_{1}}{T-t+\lambda}}(T-t) \ \ \mathrm{for \ each} \  \ t\in [T-h_{0},T],
		\end{split}
	\end{equation}
	where $\mathcal{K}_{5}:=\mathcal{K}_{5}(R,\delta)>0.$
	
	\textbf{Step 2}. In this step, the aim  is to give an upper bound for the term
	$\lambda N_{\lambda,R_{0}}(T)$ (i.e., \eqref{2.9} below).
	By Lemma~\ref{lemma-2.1}, the second equality in \eqref{u-F} and \eqref{g}, we get
	$$
	\frac{\mathrm{d}}{\mathrm{d}t}N_{\lambda,R_0}(t)\leq\left(\frac{1}{T-t+\lambda}+2\|b\|^2_{B_{R_0}}\right)N_{\lambda,R_0}(t)
	+2\|b\|^2_{B_{R_0}}+\frac{\E\int_{B_{R_{0}}}|(au+g)(x,t)|^{2}G_{\lambda}(x,t)\mathrm{d}x}
	{H(t)},
	$$
	which indicates that
	\begin{equation*}
		\begin{split}
			&\frac{\mathrm{d}}{\mathrm{d}t}\left[(T-t+\lambda)N_{\lambda,R_{0}}(t)\right]\\
			\leq&2(T-t+\lambda)\|b\|^2_{B_{R_0}}N_{\lambda,R_{0}}(t)+2(T-t+\lambda)\|b\|^2_{B_{R_0}}+(T-t+\lambda)\frac{\E\int_{B_{R_{0}}}|(au+g)(x,t)|^{2}G_{\lambda}(x,t)\mathrm{d}x}
			{H(t)}\\
			\leq&2(T-t+\lambda)\|b\|^2_{B_{R_0}}N_{\lambda,R_{0}}(t)+2(T-t+\lambda)\|b\|^2_{B_{R_0}}
			+2(T-t+\lambda)
			\left(\|a\|^2_{\infty}+\frac{\E\int_{B_{R_{0}}}|g(x,t)|^{2}G_{\lambda}(x,t)\mathrm{d}x}
			{H(t)}\right),
		\end{split}
	\end{equation*}
	this, along with Gronwall's inequality implies that
	\begin{equation*}
		\begin{split}
			\lambda N_{\lambda,R_{0}}(T)\leq&(T-t+\lambda)N_{\lambda,R_{0}}(t)e^{2\|b\|^2_{B_{R_0}}(T-t)}+2e^{2\|b\|^2_{B_{R_0}}(T-t)}
			\left(\|a\|^{2}_{\infty}+\|b\|^2_{B_{R_0}}\right)\int_{t}^{T}(T-s+\lambda)\mathrm{d}s\\
			&+2e^{2\|b\|^2_{B_{R_0}}(T-t)}\int_{t}^{T}(T-s+\lambda)\frac{\E\int_{B_{R_{0}}}|g(x,s)|^{2}G_{\lambda}(x,s)\mathrm{d}x}
			{H(s)}\mathrm{d}s.
		\end{split}
	\end{equation*}
	Hence, for any $0<T-2\varepsilon\leq t<T$ (where $\varepsilon$ will be determined later), we have
	\begin{equation}\label{3.121212}
		\begin{split}
			\frac{\lambda}{2\varepsilon+\lambda}N_{\lambda,R_{0}}(T)\leq &N_{\lambda,R_{0}}(t)e^{4\|b\|^2_{B_{R_0}}\varepsilon}+4e^{4\|b\|^2_{B_{R_0}}\varepsilon}
			\left(\|a\|^{2}_{\infty}+\|b\|^2_{B_{R_0}}\right)\varepsilon\\
			&+4e^{4\|b\|^2_{B_{R_0}}\varepsilon}\int_{t}^{T}
			\frac{\E\int_{B_{R_{0}}}|g(x,s)|^{2}G_{\lambda}(x,s)\mathrm{d}x}
			{H(s)}\mathrm{d}s,
		\end{split}
	\end{equation}
	this, along with Lemma~\ref{lemma-dH}, \eqref{u-F} and \eqref{g} implies that
	\begin{equation}\label{3.131313}
		\begin{split}
			&\displaystyle{}\frac{\mathrm{d}}{\mathrm{d}t}
			H(t)+\frac{\lambda}
			{2\varepsilon+\lambda}N_{\lambda,R_{0}}(T)H(t)\\
			\leq&e^{4\|b\|^2_{B_{R_0}}\varepsilon}\left(2\|a\|_{\infty}+\|b\|^{2}_{B_{R_0}}+4\varepsilon\|a\|^{2}_{\infty}+4\varepsilon\|b\|^{2}_{B_{R_0}}\right)
			H(t)\\
			&+H(t)e^{4\|b\|^2_{B_{R_0}}\varepsilon}\left(\frac{\E\int_{B_{R_{0}}}u(x,t)g(x,t)G_{\lambda}(x,t)\mathrm{d}x}
			{H(t)}+2\int_{t}^{T}
			\frac{\E\int_{B_{R_{0}}}|g(x,s)|^{2}G_{\lambda}(x,s)\mathrm{d}x}
			{H(s)}\mathrm{d}s\right).
		\end{split}
	\end{equation}
	
	Next, on one hand, it follows from  (\ref{3.777}) and (\ref{3.111111}) that
	\begin{equation}\label{2.7}
		\begin{split}
			&\displaystyle{}\frac{\E\int_{B_{R_{0}}}u(x,t)g(x,t)G_{\lambda}(x,t)\mathrm{d}x}
			{H(t)}+2\int_{t}^{T}
			\frac{\E\int_{B_{R_{0}}}|g(x,s)|^{2}G_{\lambda}(x,s)\mathrm{d}x}
			{H(s)}\mathrm{d}s\\
			\leq& \mathcal{K}_{6}\left(1+2\varepsilon\right)\left(1+T^{-2}\right)
			e^{-\frac{\mathcal{K}_{1}}{2\varepsilon+\lambda}}e^{\frac{C_{5}+\mathcal{K}_{1}}{h_{0}}}\\
			:=&Q_{h_{0},\varepsilon,\lambda}
			\ \ \mathrm{for\ each} \ 0<T-2\varepsilon\leq t<T \ \mathrm{ with} \  2\varepsilon\in (0,h_{0}],
		\end{split}
	\end{equation}
	where $\mathcal{K}_{6}:=\mathcal{K}_{6}(R,\delta)>0.$
	This, along with (\ref{3.131313}), implies that
	\begin{equation*}
		\begin{split}
			\frac{\mathrm{d}}{\mathrm{d}t}H(t)
			\leq -\left(\frac{\lambda}{2\varepsilon+\lambda}N_{\lambda,R_{0}}(T)-
			e^{4\|b\|^2_{B_{R_0}}\varepsilon}\left(2\|a\|_{\infty}+\|b\|^{2}_{B_{R_0}}+4\varepsilon\|a\|^{2}_{\infty}+4\varepsilon\|b\|^{2}_{B_{R_0}}
			+Q_{h_{0},\varepsilon,\lambda}\right)\right)
			H(t),
		\end{split}
	\end{equation*}
	which indicates that
	$$\frac{\mathrm{d}}{\mathrm{d}t}\left[e^{\left(\frac{\lambda}
		{2\varepsilon+\lambda}N_{\lambda,R_{0}}(T)-
		e^{4\|b\|^2_{B_{R_0}}\varepsilon}\left(2\|a\|_{\infty}+\|b\|^{2}_{B_{R_0}}+4\varepsilon\|a\|^{2}_{\infty}+4\varepsilon\|b\|^{2}_{B_{R_0}}
		+Q_{h_{0},\varepsilon,\lambda}\right)\right)t}H(t)\right]\leq 0
	$$
	for each $0<T-2\varepsilon\leq t<T$ with $2\varepsilon\in (0,h_{0}]$.
	Integrating the above inequality over $(T-2\varepsilon,T-\varepsilon)$, we obtain
	\begin{equation*}
		\begin{split}
			e^{\frac{\varepsilon\lambda}{2\varepsilon+\lambda}N_{\lambda,R_{0}}(T)}
			H(T-\varepsilon)
			\leq e^{e^{4\|b\|^2_{B_{R_0}}\varepsilon}\left(2\|a\|_{\infty}+\|b\|^{2}_{B_{R_0}}+4\varepsilon\|a\|^{2}_{\infty}+4\varepsilon\|b\|^{2}_{B_{R_0}}
				+Q_{h_{0},\varepsilon,\lambda}\right)\varepsilon}H(T-2\varepsilon).
		\end{split}
	\end{equation*}
	This yields
	\begin{equation}\label{3.151515}
		\begin{split}
			e^{\frac{\varepsilon\lambda}{2\varepsilon+\lambda}N_{\lambda,R_{0}}(T)}\leq e^{e^{4\|b\|^2_{B_{R_0}}\varepsilon}\left(2\|a\|_{\infty}+\|b\|^{2}_{B_{R_0}}+4\varepsilon\|a\|^{2}_{\infty}+4\varepsilon\|b\|^{2}_{B_{R_0}}
				+Q_{h_{0},\varepsilon,\lambda}\right)\varepsilon}
			\frac{\E\int_{B_{R_{0}}}| u(x,T-2\varepsilon)|^{2}e^{-\frac{|x-x_{0}|^{2}}{4(2\varepsilon+\lambda)}}\mathrm{d}x}{\E\int_{B_{R_{0}}}| u(x,T-\varepsilon)|^{2}e^{-\frac{|x-x_{0}|^{2}}{4(\varepsilon+\lambda)}}\mathrm{d}x}.
		\end{split}
	\end{equation}

	On the other hand, by \eqref{chi}, the first equality in \eqref{u-F}, \eqref{3.555} and noting that $e^{-\frac{|x-x_{0}|^{2}}{4(2\varepsilon+\lambda)}}\leq1$, $R_0>(1+\delta)R$ and $B_{2R_0}\subset Q_{2R_{0}}$, we see
	\begin{equation*}
		\begin{split}
			&\frac{\E\int_{B_{R_{0}}}| u(x,T-2\varepsilon)|^{2}
				e^{-\frac{|x-x_{0}|^{2}}{4(2\varepsilon+\lambda)}}\mathrm{d}x}{\E\int_{B_{R_{0}}}| u(x,T-\varepsilon)|^{2}e^{-\frac{|x-x_{0}|^{2}}{4(\varepsilon+\lambda)}}\mathrm{d}x}
			\leq\frac{e^{\frac{((1+\delta)R)^{2}}{4(\varepsilon+\lambda)}}\E\int_{B_{R_{0}}}| \varphi(x,T-2\varepsilon)|^{2}\mathrm{d}x}{\E\int_{B_{(1+\delta)R}}| \varphi(x,T-\varepsilon)|^{2}\mathrm{d}x}\\
			\leq&\frac{e^{\frac{((1+\delta)R)^{2}}{4\varepsilon}}\mathcal{K}_{2}(1+T^{-1}+\|a\|_{\infty}+\|b\|^2_{\infty})
				\E\int_{T/2}^{T}\int_{Q_{2R_{0}}}\varphi^{2}(x,t)\mathrm{d}x\mathrm{d}t}
			{\E\int_{B_{(1+\delta)R}}| \varphi(x,T-\varepsilon)|^{2}\mathrm{d}x},
		\end{split}
	\end{equation*}
	which, combined with $(ii)$ of Lemma~\ref{lemma-1.3}
	(where $r, R,\tau_{1} $ and $\tau_{2}$ are replaced by  $R, 2R_{0}, T/4$ and $T/2$, respectively),
	indicates that
	\begin{equation}\label{3.161616}
		\begin{split}
			\frac{\E\int_{B_{R_{0}}}| u(x,T-2\varepsilon)|^{2}
				e^{-\frac{|x-x_{0}|^{2}}{4(2\varepsilon+\lambda)}}\mathrm{d}x}{\E\int_{B_{R_{0}}}| u(x,T-\varepsilon)|^{2}e^{-\frac{|x-x_{0}|^{2}}{4(\varepsilon+\lambda)}}\mathrm{d}x}\leq& \displaystyle{}\frac{e^{\frac{((1+\delta)R)^{2}}{4\varepsilon}}\mathcal{K}_{2}(1+T^{-1}
				+\|a\|_{\infty}+\|b\|^2_{\infty})e^{1+\frac{C_{5}}{h_{0}}}}{e^{(2\|a\|_{\infty}+\|b\|^2_{\infty})T}}\\
			\leq&\displaystyle{}\mathcal{K}_{2}e^{\frac{((1+\delta)R)^{2}}{4\varepsilon}}\left(1+T^{-1}\right)e^{1+\frac{C_{5}}{h_{0}}}.
		\end{split}
	\end{equation}
	Then, it follows from (\ref{3.151515}) and (\ref{3.161616}) that for each  $\varepsilon\in (0,h_{0}/2]$,
	\begin{equation}\label{3.171717}
		\begin{split}
			\lambda N_{\lambda,R_{0}}(T)
			\leq&
			\displaystyle{\frac{2\varepsilon+\lambda}{\varepsilon}}
			\bigg[e^{4\|b\|^2_{B_{R_0}}\varepsilon}\left(2\|a\|_{\infty}+\|b\|^{2}_{B_{R_{0}}}+4\varepsilon\|a\|^{2}_{\infty}+4\varepsilon\|b\|^{2}_{B_{R_{0}}}
			+Q_{h_{0},\varepsilon,\lambda}\right)\varepsilon
			\\
			&+\displaystyle{\frac{(1+\delta)^2 R^2}{4\varepsilon}}+1+\displaystyle{\frac{C_{5}}{h_{0}}}+\ln\left(\mathcal{K}_{2}\left(1+T^{-1}\right)\right)\bigg].
		\end{split}
	\end{equation}
	Finally, we choose $\lambda=2\mu\varepsilon$ with $\mu\in(0,1)$ (which will be determined later) and $2\varepsilon=\mathcal{K}_{1}h_{0}/[2(C_{5}+\mathcal{K}_{1})]$ so that
	$Q_{h_{0},\varepsilon,\lambda}$ (see (\ref{2.7})) satisfies
	\begin{equation}\label{3.181818}
		\begin{split}
			Q_{h_{0},\varepsilon,\lambda}=&\mathcal{K}_{6}\left(1+2\varepsilon\right)\left(1+T^{-2}\right)
			e^{-\frac{\mathcal{K}_{1}}{2\varepsilon+\lambda}}e^{\frac{C_{5}+\mathcal{K}_{1}}{h_{0}}}
			=\mathcal{K}_{6}\left(1+2\varepsilon\right)\left(1+T^{-2}\right)e^{-\frac{2(C_5+\mathcal{K}_1)}{h_0(\mu+1)}}e^{\frac{C_{5}+\mathcal{K}_{1}}{h_{0}}}\\
			=&\mathcal{K}_{6}
			\left(1+2\varepsilon\right)\left(1+T^{-2}\right)e^{\frac{C_{5}+\mathcal{K}_{1}}{h_{0}}(\frac{\mu-1}{\mu+1})}
			\leq \mathcal{K}_{6}\left(1+2\varepsilon\right)\left(1+T^{-2}\right).
		\end{split}
	\end{equation}
	Since $2\varepsilon\leq h_{0}$, by  (\ref{3.171717}) and (\ref{3.181818}), we get
	\begin{equation}\label{3.191919}
		\begin{split}
			\lambda N_{\lambda,R_{0}}(T)\leq& 4e^{2h_0\|b\|^{2}_{\infty}}\left(h_{0}\|a\|_{\infty}+\frac{1}{2}h_0\|b\|^{2}_{\infty}+h_{0}^{2}\|a\|^{2}_{\infty}+h_0^2\|b\|^{2}_{\infty}
			+\frac{1}{2}\mathcal{K}_{6}\left(1+2\varepsilon\right)\left(1+T^{-2}\right)h_{0}\right) \\
			&+4\left[1+\frac{C_{5}}{h_{0}}+\mathcal{K}_{2}
			\left(1+T^{-1}\right)+\frac{C_{5}+\mathcal{K}_{1}}{\mathcal{K}_{1}h_{0}}(1+\delta)^2 R^2\right].
		\end{split}
	\end{equation}

	According to $(i)$ of Lemma~\ref{lemma-1.3}
	(where $r, R,\tau_{1} $ and $\tau_{2}$ are replaced by  $R, 2R_{0}, T/4$ and $T/2$, respectively),
	it is clear that
	$$h_{0}<C_{3},\ h_{0}<T,\ h_{0}T\|a\|_{\infty}<C_{3},\ h_{0}\|b\|^{2}_{\infty}<C_{3}, \ h^{3}_{0}\|b\|^{2}_{\infty}<C^{2}_{3} \ \mathrm{and} \ h^{3}_{0}\|a\|^{2}_{\infty}<C^{3}_{3}.
	$$
	These, together with (\ref{3.191919}), derive that
	\begin{equation*}
		\begin{split}
			\varepsilon\lambda N_{\lambda,R_{0}}(T)\leq&4\varepsilon e^{2h_0\|b\|^{2}_{\infty}}\left(h_{0}\|a\|_{\infty}+\frac{1}{2}h_0\|b\|^{2}_{\infty}+h_{0}^{2}\|a\|^{2}_{\infty}
			+h_0^2\|b\|^{2}_{\infty}
			+\frac{1}{2}\mathcal{K}_{6}\left(1+2\varepsilon\right)\left(1+T^{-2}\right)h_{0}\right) \\
			&+4\varepsilon\left[1+\frac{C_{5}}{h_{0}}+\mathcal{K}_{2}
			\left(1+T^{-1}\right)+\frac{C_{5}+\mathcal{K}_{1}}{\mathcal{K}_{1}h_{0}}(1+\delta)^2 R^2\right]\\
			\leq&4e^{2h_0\|b\|^{2}_{\infty}}\left(h_{0}T\|a\|_{\infty}+h_{0}^{3}\|a\|^{2}_{\infty}
			+h_0^3\|b\|^{2}_{\infty}
			+\frac{1}{2}\mathcal{K}_{6}\left(1+2h_0\right)\left(1+T^{-2}\right)h_{0}^2\right) \\
			&+4\left[h_{0}+C_{5}+\mathcal{K}_{2}h_{0}\left(1+T^{-1}\right)
			+\frac{C_{5}+\mathcal{K}_{1}}{\mathcal{K}_{1}}(1+\delta)^2 R^{2}\right]\\ 
			\leq&2e^{2C_3}\left[2C_{3}+ 2C_{3}^{3}+2C_3^2+\mathcal{K}_{6}\left(1+2C_{3}\right)\left(1+C_{3}^{2}\right) \right]\\
			&+4\left[C_{3}+C_{5}+\mathcal{K}_{2}\left(1+C_{3}\right)
			+\frac{C_{5}+\mathcal{K}_{1}}{\mathcal{K}_{1}}(1+\delta)^2 R^{2}\right].
		\end{split}
	\end{equation*}
	Hence, recalling that $\lambda=2\mu\varepsilon$, we have
	\begin{equation}\label{2.9}
		\frac{16\lambda}{r^{2}}\left(\frac{N}{4}+\frac{1}{2}\lambda N_{\lambda,R_{0}}(T)\right)
		\leq\frac{16\mu}{r^{2}}\left(\frac{N}{2}C_{3}+\varepsilon\lambda N_{\lambda,R_{0}}(T)\right)\leq2\mu(1+\mathcal{K}_{7}),
	\end{equation}
	where $\mathcal{K}_{7}:=\mathcal{K}_{7}(r,R,\delta, {N})>0.$
	
	\textbf{Step 3}. We claim that
	\begin{equation}\label{3.212121}
		\begin{split}
			\displaystyle{} \E\int_{B_{R_{0}}}| u(x,T)|^{2}e^{-\frac{|x-x_{0}|^{2}}{4\lambda}}\mathrm{d}x
			\leq\displaystyle{}\E\int_{B_{r}}|\varphi(x,T)|^{2}e^{-\frac{|x-x_{0}|^{2}}{4\lambda}}\mathrm{d}x
			\displaystyle{}+2\mu(1+\mathcal{K}_{7})\E\int_{B_{R_{0}}}|u(x,T)|^{2}e^{-\frac{|x-x_{0}|^{2}}{4\lambda}}\mathrm{d}x.
		\end{split}
	\end{equation}
	Indeed, noting that $u$ is $H^1(B_{R_0})$-value, by  \cite[page 1951]{duan2020observability}, also see \cite{escauriaza2006doubling,phung2010quantitative,phung2014bang}. we have
	\begin{equation*}
		\begin{split}
			&\frac{1}{16\lambda}\int_{B_{R_{0}}}|x-x_{0}|^{2}| u(x,T)|^{2}e^{-\frac{|x-x_{0}|^{2}}{4\lambda}}\mathrm{d}x\\
			\leq&\frac{N}{4}\int_{B_{R_{0}}}|u(x,T)|^{2}e^{-\frac{|x-x_{0}|^{2}}{4\lambda}}\mathrm{d}x+\lambda\int_{B_{R_{0}}}|\nabla u(x,T)|^{2}e^{-\frac{|x-x_{0}|^{2}}{4\lambda}}\mathrm{d}x\ \ \ \bP\mathrm{-a.s.\ in \ }B_{R_{0}}.
		\end{split}
	\end{equation*}
	This implies
	\begin{equation}\label{3.222222}
		\begin{split}
			&\displaystyle{}\E\int_{B_{R_{0}}}|u(x,T)|^{2}e^{-\frac{|x-x_{0}|^{2}}{4\lambda}}\mathrm{d}x\\
			\leq&\displaystyle{} \E\int_{B_{R_{0}}\setminus B_{r}}
			\frac{|x-x_{0}|^{2}}{r^{2}}| u(x,T)|^{2}e^{-\frac{|x-x_{0}|^{2}}{4\lambda}}\mathrm{d}x
			\displaystyle{}+\E\int_{B_{r}}|u(x,T)|^{2}e^{-\frac{|x-x_{0}|^{2}}{4\lambda}}\mathrm{d}x\\
			\leq&\displaystyle{}\frac{16\lambda}{r^{2}}
			\left[\frac{N}{4}+\frac{1}{2}\lambda N_{\lambda,R_{0}}(T)\right]\E\int_{B_{R_{0}}}|u(x,T)|^{2}
			e^{-\frac{|x-x_{0}|^{2}}{4\lambda}}\mathrm{d}x
			+\displaystyle{}\E\int_{B_{r}}|\varphi(x,T)|^{2}e^{-\frac{|x-x_{0}|^{2}}{4\lambda}}\mathrm{d}x  ,
		\end{split}
	\end{equation}
	where in the last line, we used the definition of $N_{\lambda,R_{0}}(T)$ (see \eqref{HDN})
	and the fact that $u=\varphi$ in $B_{r}$ (see \eqref{chi} and \eqref{u-F}).
	Then (\ref{3.212121}) follows from  (\ref{3.222222}) and  (\ref{2.9}) immediately.
	
	\textbf{Step 4}. End of the proof.
	We choose $\mu=1/[2(1+\mathcal{K}_{7})]$. Then, $\lambda=2\mu\varepsilon=\mathcal{K}_{1}h_{0}/[4(1+\mathcal{K}_{7})(C_{5}+\mathcal{K}_{1})].$
	By \eqref{3.212121}, $e^{-\frac{|x-x_0|^2}{4\lambda}}\leq1$ and the fact that $u=\varphi$ in $B_R$ (see \eqref{chi} and \eqref{u-F}), we have
	\begin{equation*}
		\begin{split}
			e^{-\frac{R^2}{4\lambda}}\E\int_{B_{R}}|\varphi(x,T)|^{2}\mathrm{d}x\leq& \E\int_{B_{R}}|\varphi(x,T)|^{2}e^{-\frac{|x-x_{0}|^{2}}{4\lambda}}\mathrm{d}x=
			\E\int_{B_{R_{0}}}|u(x,T)|^{2}e^{-\frac{|x-x_{0}|^{2}}{4\lambda}}\mathrm{d}x\\
			\leq&\displaystyle{}\E\int_{B_{r}}|\varphi(x,T)|^{2}e^{-\frac{|x-x_{0}|^{2}}{4\lambda}}\mathrm{d}x
			\displaystyle{}+\E\int_{B_{R_{0}}}|u(x,T)|^{2}e^{-\frac{|x-x_{0}|^{2}}{4\lambda}}\mathrm{d}x\\
			\leq&2\E\int_{B_{r}}|\varphi(x,T)|^{2}\mathrm{d}x.
		\end{split}
	\end{equation*}
	This, along with the definition of $h_{0}$ (see (\ref{1.8}),
	where $r, R,\tau_{1} $ and $\tau_{2}$ are replaced by  $R, 2R_{0}, T/4$ and $T/2$, respectively), implies that
	\begin{equation*}
		\begin{split}
			&\E\int_{B_{R}}|\varphi(x,T)|^{2}dx
			\leq2e^{\frac{(1+\mathcal{K}_{7})(C_{5}+\mathcal{K}_{1})R^{2}}
				{\mathcal{K}_{1}h_{0}}}\E\int_{B_{r}}|\varphi(x,T)|^{2}\mathrm{d}x\\
			\leq& \Big[ (1+C_{4})\left(e^{[1+2C_{1}(1+R^{-2})]
				(1+4T^{-1}+\|a\|^{2/3}_{\infty}+\|b\|^2_{\infty})+\frac{4C_{3}}{T}
				+(2\|a\|_{\infty}+\|b\|^2_{\infty})T}\right)\\
			&\frac{\E\int_{T/2}^{T}\int_{Q_{2R_{0}}}\varphi^{2}\mathrm{d}x\mathrm{d}t}
			{\E\int_{B_{R}}|\varphi(x,T)|^{2}\mathrm{d}x} \Big] ^{\frac{(1+\mathcal{K}_{7})(C_{5}+\mathcal{K}_{1})R^{2}}
				{\mathcal{K}_{1}C_{3}}}\times2\E\int_{B_{r}}|\varphi(x,T)|^{2}\mathrm{d}x.
		\end{split}
	\end{equation*}
	Hence, we can conclude that the desired estimate of Lemma~\ref{lemma-2.2} holds  with
	$$
	\gamma=\frac{(1+\mathcal{K}_{7})(C_{5}+\mathcal{K}_{1})R^{2}}
	{C_{3}\mathcal{K}_{1}+(1+\mathcal{K}_{7})(C_{5}+\mathcal{K}_{1})R^{2}}\in(0,1).
	$$

	In summary, we finish the proof of this lemma.
\end{proof}

Finally, based on Lemma \ref{lemma-2.2}, we are ready to prove Theorem \ref{2.3}.
\begin{proof}[\textbf{Proof of Theorem \ref{2.3}}]
	By Lemma~\ref{lemma-2.2} (where $r, R$ and $\delta$ are replaced by  $r, \sqrt{N}R$ and $1/2$, respectively), we obtain
	\begin{equation*}
		\begin{split}
			&\E\int_{Q_{R}(x_{i})}|\varphi(x,T)|^{2}\mathrm{d}x
			\leq \E\int_{B_{\sqrt{N}R}(x_{i})}|\varphi(x,T)|^{2}\mathrm{d}x\\
			\leq&\bigg[\mathcal{\widehat{K}}_{1}e^{[1+2C_{1}(1+R^{-2})](1+4T^{-1}+
				\|a\|^{2/3}_{\infty}+\|b\|_{\infty}^2)+\mathcal{\widehat{K}}_{2}T^{-1}
				+(2\|a\|_{\infty}+\|b\|_{\infty}^2)T}\E\int_{T/2}^{T}\int_{Q_{4\sqrt{N}R}(x_{i})}\varphi^{2}\mathrm{d}x\mathrm{d}t\bigg]^{\theta}\\
			&\times\left[2\E\int_{B_{r}(x_{i})}|\varphi(x,T)|^{2}\mathrm{d}x\right]^{1-\theta},
		\end{split}
	\end{equation*}
	where $\mathcal{\widehat{K}}_{1}:=\mathcal{\widehat{K}}_{1}(R)>0,
	\mathcal{\widehat{K}}_{2}:=\mathcal{\widehat{K}}_{2}(R)>0$ and
	$\theta:=\theta(r,R)\in (0,1).$ This, along with Young's  inequality,
	implies that for each $\varepsilon>0,$
	\begin{equation*}
		\begin{split}
			\E\int_{Q_{R}(x_{i})}|\varphi(x,T)|^{2}\mathrm{d}x
			\leq&\varepsilon\theta\mathcal{\widehat{K}}_{1}e^{[1+2C_{1}(1+R^{-2})]
				(1+4T^{-1}+\|a\|^{2/3}_{\infty}+\|b\|^{2}_{\infty})+\mathcal{\widehat{K}}_{2}T^{-1}+(2\|a\|_{\infty}+\|b\|^{2}_{\infty})T}\\
			&\times\E\int_{T/2}^{T}\int_{Q_{4\sqrt{N}R}(x_{i})}\varphi^{2}\mathrm{d}x\mathrm{d}t
			+2\varepsilon^{-\frac{\theta}{1-\theta}}(1-\theta)\E\int_{B_{r}(x_{i})}
			|\varphi(x,T)|^{2}\mathrm{d}x.
		\end{split}
	\end{equation*}
	Then
	\begin{equation}\label{3.44444}
		\begin{split}
			&\displaystyle{}\E\int_{\mathbb{R}^{N}}|\varphi(x,T)|^{2}\mathrm{d}x
			=\sum_{i\geq1}\E\int_{Q_{R}(x_{i})}|\varphi(x,T)|^{2}\mathrm{d}x\\
			\leq&\displaystyle{}\varepsilon\theta\mathcal{\widehat{K}}_{1}
			e^{[1+2C_{1}(1+R^{-2})]
				(1+4T^{-1}+\|a\|^{2/3}_{\infty}+\|b\|^{2}_{\infty})+\mathcal{\widehat{K}}_{2}T^{-1}+(2\|a\|_{\infty}+\|b\|^{2}_{\infty})T}\\
			&\sum_{i\geq 1}\E\int_{T/2}^{T}\int_{Q_{4\sqrt{N}R}(x_{i})}\varphi^{2}\mathrm{d}x\mathrm{d}t\displaystyle{}+2\varepsilon^{-\frac{\theta}{1-\theta}}(1-\theta)\E\int_{\omega}|\varphi(x,T)|^{2}\mathrm{d}x.
		\end{split}
	\end{equation}
	Since
	\begin{equation*}
		\sum_{i\geq1}\E\int_{T/2}^{T}\int_{Q_{4\sqrt{N}R}(x_{i})}\varphi^{2}
		\mathrm{d}x\mathrm{d}t\leq \mathcal{\widehat{K}}_{3} \E\int_{T/2}^{T}
		\int_{\mathbb{R}^{N}}\varphi^{2}\mathrm{d}x\mathrm{d}t,
	\end{equation*}
	where $\mathcal{\widehat{K}}_{3}>0$, it follows from (\ref{3.44444}) that
	\begin{equation*}
		\begin{split}
			\E\int_{\mathbb{R}^{N}}|\varphi(x,T)|^{2}\mathrm{d}x
			\leq&\varepsilon\theta\mathcal{\widehat{K}}_{1}\mathcal{\widehat{K}}_{3}
			e^{[1+2C_{1}(1+R^{-2})]
				(1+4T^{-1}+\|a\|^{2/3}_{\infty}+\|b\|^{2}_{\infty})+\mathcal{\widehat{K}}_{2}T^{-1}+(2\|a\|_{\infty}+\|b\|^{2}_{\infty})T}
			\\ &\times\E\int_{T/2}^{T}\int_{\mathbb{R}^{N}}\varphi^{2}\mathrm{d}x\mathrm{d}t+2\varepsilon^{-\frac{\theta}{1-\theta}}(1-\theta)
			\E\int_{\omega}|\varphi(x,T)|^{2}\mathrm{d}x  \ \mathrm{for \ each} \ \varepsilon>0.
		\end{split}
	\end{equation*}
	This implies
	\begin{equation}\label{3.444440}
		\begin{split}
			&\displaystyle{}\E\int_{\mathbb{R}^{N}}|\varphi(x,T)|^{2}\mathrm{d}x\\
			\leq&\displaystyle{}\left[\mathcal{\widehat{K}}_{1}\mathcal{\widehat{K}}_{3}
			e^{[1+2C_{1}(1+R^{-2})]
				(1+4T^{-1}+\|a\|^{2/3}_{\infty}+\|b\|^{2}_{\infty})+\mathcal{\widehat{K}}_{2}T^{-1}+(2\|a\|_{\infty}+\|b\|^{2}_{\infty})T}
			\E\int_{T/2}^{T}\int_{\mathbb{R}^{N}}\varphi^{2}\mathrm{d}x\mathrm{d}t\right]^{\theta}\\
			&\displaystyle{}\times\left[2\E\int_{\omega}|\varphi(x,T)|^{2}\mathrm{d}x\right]^{1-\theta}.
		\end{split}
	\end{equation}

	Applying the \Ito formula to $\varphi^2$ and then taking expectation and by Gronwall's inequality,  we have the energy estimate of the equation \eqref{1.1}:
	\begin{equation}\label{energy}
		\E\int_{\mathbb{R}^N} |\varphi(x,t)|^2\,\mathrm dx\leq e^{(2\|a\|_{\infty}+\|b\|_{\infty}^2)t}
		\E\int_{\mathbb{R}^N} |\varphi_0(x)|^2\,\mathrm dx\  \ \ \mathrm{for \  each} \  t\in [0,T].
	\end{equation}
	Finally, by (\ref{3.444440}), we deduce
	\begin{equation*}
		\begin{split}
			\E\int_{\mathbb{R}^{N}}|\varphi(x,T)|^{2}\mathrm{d}x
			\leq&\bigg[\mathcal{\widehat{K}}_{1}\mathcal{\widehat{K}}_{3}
			Te^{[1+2C_{1}(1+R^{-2})](1+4T^{-1}+\|a\|^{2/3}_{\infty}+\|b\|_{\infty}^2)
				+\mathcal{\widehat{K}}_{2}T^{-1}+2(2\|a\|_{\infty}+\|b\|_{\infty}^2)T}
			\\
			&\times\E\int_{\mathbb{R}^{N}}|\varphi_{0}(x)|^{2}\mathrm{d}x\bigg]^{\theta}\times\left[2\E\int_{\widetilde{\omega}}|\varphi(x,T)|^{2}\mathrm{d}x\right]^{1-\theta}.
		\end{split}
	\end{equation*}
	Hence, (\ref{3.33333}) follows from the latter inequality immediately.
\end{proof}

\section{Proof\ of\ Corollary~\ref{Thm1}}\label{pro}
Now, we are able to present the proof of Corollary~\ref{Thm1} by Theorem \ref{2.3} and the telescoping series method (see \cite{apraiz2014observability,phung2013observability}). For the convenience of the reader, we provide here the detailed computation.

\noindent\textbf{Proof\ of\ Corollary~\ref{Thm1}}.
For any $0\leq t_{1}<t_{2}\leq T$, by using Theorem~\ref{2.3}, we obtain from Young's inequality that
\begin{equation}\label{2019-7-9}
	\E\|\varphi(t_{2})\|^{2}_{L^{2}(\mathbb{R}^{N})}\leq\varepsilon
	\E\|\varphi(t_{1})\|^{2}_{L^{2}(\mathbb{R}^{N})}+
	\frac{\mathcal{\widetilde{K}}_{1}}{\varepsilon^{\alpha}}e^{\frac{\mathcal{\widetilde{K}}_{2}}{t_{2}-t_{1}}}
	\E\|\varphi(t_{2})\|^{2}_{L^{2}(\omega)} \ \ \ \mathrm{for\ each}\ \varepsilon>0,
\end{equation}
where $\mathcal{\widetilde{K}}_{1}:= e^{\frac{C_8}{1-\theta}
	\left(T+T(\|a\|_{\infty}+\|b\|_{\infty}^2)+\|a\|_{\infty}^{2/3}+\|b\|_{\infty}^2\right)},$
$\mathcal{\widetilde{K}}_{2}:= C_{8}/(1-\theta)$
and $\alpha:=\theta/(1-\theta)$.

Let $l$ be a density point of $E$. According to Proposition 2.1 in \cite{phung2013observability},
for each $\kappa>1$, there exists $l_{1}\in (l,T)$, depending on $\kappa$ and $E$,
so that the sequence $\{l_{m}\}_{m\geq1}$, given by
$$
l_{m+1}=l+\frac{1}{\kappa^{m}}(l_{1}-l),
$$
satisfies
\begin{equation}\label{3.2525251}
	l_{m}-l_{m+1}\leq 3|E\cap(l_{m+1},l_{m})|.
\end{equation}

Next, let $0<l_{m+2}<l_{m+1}\leq t<l_{m}<l_{1}<T$. It follows from (\ref{2019-7-9}) (where $t_1$, $t_2$ are replaced by $l_{m+2}$ and $t$, respectively) that
\begin{equation}\label{3.2525252}
	\E\|\varphi(t)\|^{2}_{L^{2}(\mathbb{R}^{N})}\leq \varepsilon\E\|\varphi(l_{m+2})\|^{2}_{L^{2}(\mathbb{R}^{N})}
	+\frac{\mathcal{\widetilde{K}}_{1}}{\varepsilon^{\alpha}}
	e^{\frac{\mathcal{\widetilde{K}}_{2}}{t-l_{m+2}}}\E\|\varphi(t)\|^{2}_{L^{2}(\omega)} \ \mathrm{for\ each}\ \varepsilon>0.
\end{equation}
Similar to \eqref{energy},  we have
$$
\E\|\varphi(l_{m})\|_{L^{2}(\mathbb{R}^{N})}\leq e^{(2\|a\|_{\infty}+\|b\|_{\infty}^2)T}\E\|\varphi(t)\|_{L^{2}(\mathbb{R}^{N})}.
$$
This, along with (\ref{3.2525252}), implies for each $\varepsilon>0$,
$$
\E\|\varphi(l_{m})\|^{2}_{L^{2}(\mathbb{R}^{N})}\leq e^{(2\|a\|_{\infty}+\|b\|_{\infty}^2)T}\left(\varepsilon\E\|\varphi(l_{m+2})
\|^{2}_{L^{2}(\mathbb{R}^{N})}+
\frac{\mathcal{\widetilde{K}}_{1}}{\varepsilon^{\alpha}}
e^{\frac{\mathcal{\widetilde{K}}_{2}}{t-l_{m+2}}}\E\|\varphi(t)\|^{2}_{L^{2}(\omega)}\right),
$$
which indicates that
$$\E\|\varphi(l_{m})\|^{2}_{L^{2}(\mathbb{R}^{N})}
\leq \varepsilon\E\|\varphi(l_{m+2})\|^{2}_{L^{2}(\mathbb{R}^{N})}
+\frac{\mathcal{\widetilde{K}}_{3}}{\varepsilon^{\alpha}}
e^{\frac{\mathcal{\widetilde{K}}_{2}}{t-l_{m+2}}}\E\|\varphi(t)\|^{2}_{L^{2}(\omega)}
\ \mathrm{for\ each}\ \varepsilon>0,
$$
where $\mathcal{\widetilde{K}}_{3}=(e^{(2\|a\|_{\infty}+\|b\|_{\infty}^2)T})^{1+\alpha}\mathcal{\widetilde{K}}_{1}$.
Integrating the latter inequality over $ E\cap(l_{m+1},l_{m})$ gives
\begin{equation}\label{3.2525253}
	\begin{split}
		\displaystyle{}|E\cap(l_{m+1},l_{m})|\E\|\varphi(l_{m})\|^{2}_{L^{2}(\mathbb{R}^{N})}
		\leq&\displaystyle{}\varepsilon |E\cap(l_{m+1},l_{m})|\E\|\varphi(l_{m+2})\|^{2}_{L^{2}(\mathbb{R}^{N})}\\
		&\displaystyle{}+\frac{\mathcal{\widetilde{K}}_{3}}
		{\varepsilon^{\alpha}}e^{\frac{\mathcal{\widetilde{K}}_{2}}{l_{m+1}-l_{m+2}}}
		\E\int_{l_{m+1}}^{l_{m}}\chi_{E}\|\varphi(t)\|^{2}_{L^{2}(\omega)}\mathrm{d}t
		\;\ \mathrm{for\ each}\ \varepsilon>0.
	\end{split}
\end{equation}
Here and in the sequel, $\chi_{E}$ denotes the characteristic function of $E$.

Since $l_{m}-l_{m+1}=(\kappa-1)(l_{1}-l)/\kappa^{m},$ by (\ref{3.2525253}) and (\ref{3.2525251}), we obtain
\begin{equation*}
	\begin{split}
		\E\|\varphi(l_{m})\|^{2}_{L^{2}(\mathbb{R}^{N})}
		\leq& \frac{1}{|E\cap(l_{m+1},l_{m})|}
		\frac{\mathcal{\widetilde{K}}_{3}}{\varepsilon^{\alpha}}
		e^{\frac{\mathcal{\widetilde{K}}_{2}}{l_{m+1}-l_{m+2}}}
		\E\int_{l_{m+1}}^{l_{m}}\chi_{E}\|\varphi(t)\|^{2}_{L^{2}(\omega)}\mathrm{d}t+\varepsilon \E\|\varphi(l_{m+2})\|^{2}_{L^{2}(\mathbb{R}^{N})}\\
		\leq&\frac{3\kappa^{m}}{(l_{1}-l)(\kappa-1)}
		\frac{\mathcal{\widetilde{K}}_{3}}{\varepsilon^{\alpha}}
		e^{\mathcal{\widetilde{K}}_{2}\left(\frac{1}{l_{1}-l}\frac{\kappa^{m+1}}{\kappa-1}\right)}
		\E\int_{l_{m+1}}^{l_{m}}\chi_{E}\|\varphi(t)\|^{2}_{L^{2}(\omega)}\mathrm{d}t+
		\varepsilon \E\|\varphi(l_{m+2})\|^{2}_{L^{2}(\mathbb{R}^{N})}
	\end{split}
\end{equation*}
for each $\varepsilon>0$.
This yields
\begin{equation}\label{3.2525254}
	\begin{split}
		\displaystyle{}\E\|\varphi(l_{m})\|^{2}_{L^{2}(\mathbb{R}^{N})}\leq& \displaystyle{}
		\frac{1}{\varepsilon^{\alpha}}\frac{3}{\kappa}
		\frac{\mathcal{\widetilde{K}}_{3}}{\mathcal{\widetilde{K}}_{2}}
		e^{2\mathcal{\widetilde{K}}_{2}\left(\frac{1}{l_{1}-l}
			\frac{\kappa^{m+1}}{\kappa-1}\right)}\E\int_{l_{m+1}}^{l_{m}}\chi_{E}
		\|\varphi(t)\|^{2}_{L^{2}(\omega)}\mathrm{d}t\displaystyle{}+
		\varepsilon \E\|\varphi(l_{m+2})\|^{2}_{L^{2}(\mathbb{R}^{N})}
	\end{split}
\end{equation}
for each $\varepsilon>0$.
Denote by $d:= 2\mathcal{\widetilde{K}}_{2}/[\kappa(l_{1}-l)(\kappa-1)]$.
It follows from \eqref{3.2525254} that
\begin{equation*}
	\begin{split}
		\varepsilon^{\alpha}e^{-d\kappa^{m+2}}\E\|\varphi(l_{m})\|^{2}_{L^{2}(\mathbb{R}^{N})}
		-\varepsilon^{1+\alpha}e^{-d\kappa^{m+2}}\E\|\varphi(l_{m+2})\|^{2}_{L^{2}(\mathbb{R}^{N})}
		\leq\frac{3}{\kappa}\frac{\mathcal{\widetilde{K}}_{3}}
		{\mathcal{\widetilde{K}}_{2}}\E\int_{l_{m+1}}^{l_{m}}\chi_{E}\|\varphi(t)\|^{2}_{L^{2}(\omega)}\mathrm{d}t
	\end{split}
\end{equation*}
for each $\varepsilon>0$.

Choosing $\varepsilon=e^{-d\kappa^{m+2}}$ in the above inequality gives
\begin{equation}\label{3.25252555}
	\begin{split}
		\displaystyle{}e^{-(1+\alpha)d\kappa^{m+2}}\E\|\varphi(l_{m})\|^{2}_{L^{2}(\mathbb{R}^{N})}
		-e^{-(2+\alpha)d\kappa^{m+2}}\E\|\varphi(l_{m+2})\|^{2}_{L^{2}(\mathbb{R}^{N})}
		\leq\displaystyle{}\frac{3}{\kappa}\frac{\mathcal{\widetilde{K}}_{3}}
		{\mathcal{\widetilde{K}}_{2}}\E\int_{l_{m+1}}^{l_{m}}\chi_{E}\|\varphi(t)\|^{2}_{L^{2}(\omega)}\mathrm{d}t.
	\end{split}
\end{equation}
Taking $\kappa=\sqrt{(\alpha+2)/(\alpha+1)}$ in \eqref{3.25252555}, we then have
\begin{eqnarray*}
	e^{-(2+\alpha)d\kappa^{m}}\E\|\varphi(l_{m})\|^{2}_{L^{2}(\mathbb{R}^{N})}
	-e^{-(2+\alpha)d\kappa^{m+2}}\E\|\varphi(l_{m+2})\|^{2}_{L^{2}(\mathbb{R}^{N})}
	\leq \frac{3}{\kappa}\frac{\mathcal{\widetilde{K}}_{3}}{\mathcal{\widetilde{K}}_{2}}
	\E\int_{l_{m+1}}^{l_{m}}\chi_{E}\|\varphi(t)\|^{2}_{L^{2}(\omega)}\mathrm{d}t.
\end{eqnarray*}
Changing $m$ to $2m'$ and summing the above inequality from $m'=1$ to infinity give the desired result. Indeed,
\begin{equation*}
	\begin{split}
		&e^{-(2\|a\|_{\infty}+\|b\|_{\infty}^2)T}e^{-(2+\alpha)d\kappa^{2}}\E\|\varphi(T)\|^{2}_{L^{2}(\mathbb{R}^{N})}
		\leq e^{-(2+\alpha)d\kappa^{2}}\E\|\varphi(l_{2})\|^{2}_{L^{2}(\mathbb{R}^{N})}\\
		\leq&\sum_{m'=1}^{+\infty}\left(e^{-(2+\alpha)d\kappa^{2m'}}\E\|\varphi(l_{2m'})\|_{L^{2}(\mathbb{R}^{N})}
		-e^{-(2+\alpha)d\kappa^{2m'+2}}\E\|\varphi(l_{2m'+2})\|^{2}_{L^{2}(\mathbb{R}^{N})}\right)\\
		\leq& \frac{3}{\kappa}\frac{\mathcal{\widetilde{K}}_{3}}{\mathcal{\widetilde{K}}_{2}}\sum_{m'=1}^{+\infty}
		\E\int_{l_{2m'+1}}^{l_{2m'}}\chi_{E}\|\varphi(t)\|^{2}_{L^{2}(\omega)}\mathrm{d}t
		\leq \frac{3}{\kappa}\frac{\mathcal{\widetilde{K}}_{3}}
		{\mathcal{\widetilde{K}}_{2}}\E\int_{0}^{T}\chi_{E}\|\varphi(t)\|^{2}_{L^{2}(\omega)}\mathrm{d}t.
	\end{split}
\end{equation*}

In summary, we finish the proof of Corollary~\ref{Thm1}.
\qed

\section{Further comments}

\subsection{Controllability for the backward stochastic parabolic equation}
One could obtain the null controllability result for the backward stochastic parabolic equations by the classical duality argument as in \cite[Theorem 2.2]{tang2009null} or \cite[Theorem 1.12]{lu2015unique}. 

Given $ T>0 $, consider the following controlled backward stochastic heat equation
\begin{equation}\label{bspde}
	\left\{
	\begin{split}
		&\mathrm{d}y+\Delta y \mathrm{d}t=a_1 y \mathrm{d}t+ b_1 Y \mathrm{d}t + \chi_{E} \chi_{\omega} u \mathrm{d}t + Y \mathrm{d}W(t),   \quad &&\mathrm{in}\  \mathbb{R}^{N}\times (0,T),\\
		&y(T)=y_{T}, \quad
		&&\mathrm{in}\ \mathbb{R}^{N}.
	\end{split}\right.
\end{equation}
Here $ y_T \in L^2_{\cF_T}(\Omega;L^2(\mathbb{R}^{N})) $, $a_1 \in L^\infty_{\F}(0,+\infty;L^\infty(\R^N))$, $b_1 \in L^\infty_\F(0,+\infty;W^{1,\infty}(\R^N))$ and $ u \in L^2_{\F}(0,+\infty;\\L^2(\R^N)) $ is the control.
According to \cite[Theorem 4.10]{lv2021mathematical}, the system \eqref{bspde} has a unique solution $ (y(\cdot),Y(\cdot)) \in L^2_{\F}(\Omega;C([0,T];L^2(\mathbb{R}^{N}))) \cap L^2_{\F}(0,T;H_0^1 (\mathbb{R}^{N})) \times L^2_{\F}(0,T;L^2 (\mathbb{R}^{N})) $. 

We say system \eqref{bspde} is null controllable if for any $ y_T \in L^2_{\cF_T}(\Omega;L^2(\mathbb{R}^{N})) $, there exists a control $ u \in L^2_{\F}(0,+\infty;L^2(\R^N)) $ such that the solution of the system \eqref{bspde} with terminal state $ y_T $ and control $ u $ satisfying that $ y(0)=0 $. 
We have the following result.
\begin{corollary}
	Under the assumption of Theorem~\ref{2.3}, the system \eqref{bspde} is null controllable.
\end{corollary}
\begin{proof}
	Consider the following equation:
	\begin{equation}\label{dbspde}
		\left\{ \begin{array}{lll}
			\mathrm{d}\hat{y}-\Delta\hat{y} \mathrm{d}t=-a_1 \hat{y} \mathrm{d}t - b_1 \hat{y} \mathrm{d}W(t),    &\mathrm{in}\  \mathbb{R}^{N}\times (0,T),\\
			\hat{y}(0)=\hat{y}_{0}\in L^2_{\cF_0}(\Omega;L^2(\mathbb{R}^{N}))  &\mathrm{in}\ \mathbb{R}^{N}.
		\end{array}\right.
	\end{equation}
	We introduce a linear subspace of $ L^2_{\mathbb{F}}(0,T;L^2(\omega)) $:
	$$ \mathcal{X}\triangleq\left\{\hat{y}|_{\omega\times E}: \hat{y} \text{ solves equation }\eqref{dbspde}\right\}, $$
	and define a linear functional $ \mathcal{L} $ on $ \mathcal{X} $ as follows:
	$$ \mathcal{L}(\hat{y}|_{\omega\times E}) = -\E\int_{\R^N} \hat{y}(T)y_T \mathrm{d}x.$$
	By Corollary~\ref{Thm1}, we have that
	\begin{equation*}
		\begin{split}
			|\mathcal{L}(\hat{y}|_{\omega\times E})| &\leqslant \|\hat{y}(T)\|_{L^2_{\cF_T}(\Omega;L^2(\mathbb{R}^{N}))}\|y_T \|_{L^2_{\cF_T}(\Omega;L^2(\mathbb{R}^{N}))}\\
			&\leqslant e^{\widetilde{C}_1 }e^{C_1 \left(T +T (\|a_1 \|_{\infty}+\|b_1 \|_{\infty}^2)+\|a_1 \|_{\infty}^{2/3}+\|b_1 \|^2_{\infty}+1\right)} \|y_T \|_{L^2_{\cF_T}(\Omega;L^2(\mathbb{R}^{N}))} 
			\left(\E\int_{\omega\times E}|\hat{y}(x,t)|^{2}\mathrm dx\mathrm dt\right)^{\frac{1}{2}}.
		\end{split}
	\end{equation*}
	Therefore, $ \mathcal{L} $ is a bounded linear functional on $ \mathcal{X} $.
	By the Hahn–Banach theorem, $\mathcal{L} $ can be extended to a bounded linear functional with the same norm on $ L^2_{\mathbb{F}}(0,T;L^2(\omega)) $.  
	For simplicity, we use the same notation for this extension. 
	By the Riesz representation theorem, there exists a stochastic process $ \hat{u} \in L^2_{\mathbb{F}}(0,T;L^2(\omega)) $ such that 
	\begin{equation}\label{riesz}
		\E\int_{\omega\times E} \hat{y}\hat{u} \mathrm dx\mathrm dt = \E\int_{\R^N} \hat{y}(T)y_T \mathrm{d}x.
	\end{equation}
	Let \begin{equation*}
		u(x,t) = \left\{\begin{split}
			&\hat{u}(x,t),\quad &&(x,t) \in \omega\times E, \\ &0, \quad &&\text{else}.
		\end{split}\right.
	\end{equation*}
	Then it is obvious that $ u \in L^2_{\mathbb{F}}(0,+\infty;L^2(\R^N)) $, and we claim that this $ u $ is the control we need. In fact, for any $ y_T \in L^2_{\cF_T}(\Omega;L^2(\mathbb{R}^{N})) $, for the solution $ \hat{y} $ of equation \eqref{dbspde} and the solution $ (y, Y) $ of equation \eqref{bspde}, by the \Ito formula, we have that 
	\begin{equation}\label{riesz1}
		\begin{split}
			\E \int_{\R^N}\hat{y}(T)y(T)\mathrm{d}x&- \E \int_{\R^N}\hat{y}_0 y(0)\mathrm{d}x \\
			&= \E \int_0^T \int_{\R^N}\left[\hat{y}(-\Delta y+a_1 y+b_1 Y + \chi_{E}\chi_{\omega}u)+y(\Delta \hat{y}-a_1 \hat{y})-b_1 \hat{y}Y\right] \mathrm{d}x\mathrm{d}t\\
			&= \E \int_0^T \int_{\R^N} \hat{y}\chi_{E}\chi_{\omega}u \mathrm{d}x\mathrm{d}t\\
			&= \E\int_{\omega\times E} \hat{y}\hat{u} \mathrm dx\mathrm dt.
		\end{split}
	\end{equation}
	Combining \eqref{riesz} and \eqref{riesz1}, we get that 
	\begin{equation*}
		\E \int_{\R^N}\hat{y}_0 y(0)\mathrm{d}x = 0.
	\end{equation*}
	Since $ \hat{y}_0 $ can be chosen arbitrarily, we know that $ y(0) = 0, \bP\mathrm{-a.s.\ in \ }\R^N $.
\end{proof}

\subsection{Controllability for the forward stochastic parabolic equation}
The observability inequality for the solution of forward stochastic parabolic equation we obtained here cannot imply the controllability result for the same forward stochastic parabolic equation, because the solutions of the forward and backward stochastic parabolic equations are not equivalent. 
In fact, the concept of controllability for the forward stochastic parabolic equation is much more complicated than the deterministic couterpart, which usually involves a control in the diffusion term of the equation.
For this topic, we refer \cite{Barbu03, tang2009null, lu2011some, hernandez23} to the interesting reader.

\section{Appendix}
\noindent\textbf{Proof\ of\ Lemma~\ref{lemma-1.1}}.
For simplicity, we may write  $B_{r}:=B_{r}(x_{0})$ and  $B_{R}:= B_{R}(x_{0}).$
Let $\eta\in C_{0}^{\infty}(B_{R})$ verifies
\begin{equation}\label{2.21111}
	0\leq\eta(\cdot)\leq 1 \ \mathrm{in} \ B_{R}, \ \eta(\cdot)=1 \ \mathrm{in} \ B_{r} \
	\mathrm{and} \  |\nabla \eta(\cdot)|\leq C(R-r)^{-1}.
\end{equation}
Here and throughout the proof of Lemma~\ref{lemma-1.1},
$C$ denotes a generic positive constant. Let $\xi\in C^{\infty}(\mathbb{R})$ satisfy
\begin{equation}\label{2.31111}
	0\leq\xi(\cdot)\leq1, \ |\xi'(\cdot)|\leq C(\tau_{2}-\tau_{1})^{-1} \ \ \
	\mathrm{in} \ \ \mathbb{R},
\end{equation}
\begin{equation}\label{2.41111}
	\xi(\cdot)=0 \ \mathrm{in} \ (-\infty, T-\tau_{2}] \ \mathrm{and} \ \xi(\cdot)=1 \
	\mathrm{in} \ [T-\tau_{1}, +\infty).
\end{equation}
Applying the \Ito formula to $\eta^{2}\xi^{2}\varphi^2$, we have
\begin{equation*}
	\begin{split}
		d(\eta^{2}\xi^{2}\varphi^2)
		&=  2\xi\xi'\eta^{2}\varphi^2dt+2\eta^{2}\xi^{2}\varphi\cdot\left[\Delta\varphi dt+a\varphi dt+ b\varphi dW(t)\right]+\eta^{2}\xi^{2}b^2\varphi^2dt.
	\end{split}
\end{equation*}
Integrating the above equality over  $B_{R}\times(T-\tau_{2},t)$ for  $t\in [T-\tau_{1},T]$ and taking the expectation, noting that $\xi(T-\tau_2)=0$,   we obtain that
\begin{equation}\label{2.51111}
	\begin{split}
		\E\int_{B_{R}}\eta^{2}\xi^{2}(t)\varphi^{2}(x,t) \mathrm{d}x=&\E\int_{T-\tau_2}^t\int_{B_{R}}\left[2\xi\xi'\eta^{2}\varphi^2+2\eta^{2}\xi^{2}\varphi\cdot(\Delta\varphi +a\varphi)+\eta^{2}\xi^{2}b^2\varphi^2   \right]\mathrm{d}x\mathrm{d}s\\
		=&2\E\int_{T-\tau_2}^t\int_{B_{R}}\xi\xi'\eta^{2}\varphi^2\mathrm{d}x\mathrm{d}s
		+2\E\int_{T-\tau_2}^t\int_{B_{R}}\eta^{2}\xi^{2}\varphi\cdot\Delta\varphi\mathrm{d}x\mathrm{d}s\\
		&+\E\int_{T-\tau_2}^t\int_{B_{R}}\left(2a\eta^{2}\xi^{2}\varphi^2+\eta^{2}\xi^{2}b^2\varphi^2\right)\mathrm{d}x\mathrm{d}s.
	\end{split}
\end{equation}
Notice that
\begin{equation*}
	\begin{split}
		2\E\int_{T-\tau_2}^t\int_{B_{R}}\eta^{2}\xi^{2}\varphi\cdot\Delta\varphi\mathrm{d}x\mathrm{d}s=-4\E\int_{T-\tau_2}^t\int_{B_{R}}\xi^2\eta\varphi
		\nabla\eta\cdot\nabla\varphi\mathrm{d}x\mathrm{d}s-2\E\int_{T-\tau_2}^t\int_{B_{R}}\eta^2\xi^2|\nabla\varphi|^2\mathrm{d}x\mathrm{d}s,
	\end{split}
\end{equation*}
and by \eqref{2.51111} and Young's inequality, we have
\begin{equation}\label{2.511112}
	\begin{split}
		&\E\int_{B_{R}}\eta^{2}\xi^{2}(t)\varphi^{2}(x,t) \mathrm{d}x+\E\int_{T-\tau_2}^t\int_{B_{R}}\eta^2\xi^2|\nabla\varphi|^2\mathrm{d}x\mathrm{d}s\\
		\leq&4\E\int_{T-\tau_{2}}^{t}\int_{B_{R}} |\nabla\eta|^{2}\xi^{2}\varphi^{2} \mathrm{d}x\mathrm{d}s
		+2\E\int_{T-\tau_{2}}^{t}\int_{B_{R}} \eta^{2}\xi\xi'\varphi^{2} \mathrm{d}x\mathrm{d}s\\
		&+\E\int_{T-\tau_2}^t\int_{B_{R}}\left(2a\eta^{2}\xi^{2}\varphi^2+\eta^{2}\xi^{2}b^2\varphi^2\right)\mathrm{d}x\mathrm{d}s,
	\end{split}
\end{equation}
This, along with  (\ref{2.21111})-(\ref{2.41111}), implies that
\begin{equation*}
	\begin{split}
		&\E\int_{B_{r}}\varphi^{2}(x,t) \mathrm{d}x+\E\int_{T-\tau_{1}}^{t}\int_{B_{r}} |\nabla\varphi|^{2} \mathrm{d}x\mathrm{d}s\\
		\leq&C\left[(R-r)^{-2}+(\tau_{2}-\tau_{1})^{-1}+\|a\|_{\infty}+\|b\|_{\infty}^2\right]
		\E\int_{T-\tau_{2}}^{T}\int_{B_{R}}\varphi^{2} \mathrm{d}x\mathrm{d}s,\ \  \mathrm{for\ each} \ t\in [T-\tau_{1},T].
	\end{split}
\end{equation*}
Hence, (\ref{1.2}) follows from the last inequality immediately.
\qed

\medskip
\noindent\textbf{Proof\ of\ Lemma~\ref{lemma-1.2}}.
For each $r'>0,$ we write $B_{r'}:= B_{r'}(x_{0})$.  Let $\eta\in C_{0}^{\infty}(B_{4R/3})$ satisfies
\begin{equation}\label{2.61111}
	0\leq \eta(\cdot)\leq 1,\ |\nabla\eta(\cdot)|\leq CR^{-1}, \ |\Delta \eta(\cdot)|\leq CR^{-2}\ \mathrm{ in} \ B_{4R/3}
\end{equation}
and
\begin{equation}\label{2.61112}
	\eta(\cdot)=1 \ \mathrm{ in} \ B_{R}.
\end{equation}
Here and throughout the
proof of Lemma~\ref{lemma-1.2}, $C$ denotes a generic positive constant.
Let $\xi\in C^{\infty}(\mathbb{R})$ verifies
\begin{equation}\label{2.61113}
	0\leq \xi(\cdot)\leq1,\ |\xi'(\cdot)|\leq C\tau^{-1} \ \mathrm{ in} \ \mathbb{R},
\end{equation}
\begin{equation}\label{2.61114}
	\xi(\cdot)=0 \ \mathrm{in} \ (-\infty, T-4\tau/3] \ \mathrm{and } \ \xi(\cdot)=1  \ \mathrm{in} \ [T-\tau, +\infty).
\end{equation}

Applying the \Ito formula to $\frac{1}{2}\eta^{2}\xi^{2}\varphi^2_i$, where $ \varphi_i = \partial_{x_i} \varphi $,
integrating over  $B_{4R/3 }\times(T-4\tau/3,t)$ for  $t\in [T-\tau,T]$, taking the expectation, and noting that $\xi(T-4\tau/3)=0$, $ \eta(\cdot)\equiv0 \;\text{on} \;\partial B_{4R/3}$. Similar to the calculation of \eqref{2.511112}, we obtain that
\begin{equation}
	\begin{split}
		&\E\int_{B_{4R/3}}\eta^{2}(x)\xi^{2}(t)\varphi^{2}_i (x,t) \mathrm{d}x+\E\int_{T-4\tau/3}^t\int_{B_{4R/3}}(\xi \eta \nabla \varphi_i ) ^2 \mathrm{d}x\mathrm{d}s\\
		&\leq 2\E\int_{T-4\tau/3}^t\int_{B_{4R/3}}\xi\xi'\eta^{2}\varphi^2_i \mathrm{d}x\mathrm{d}s+4\E \int_{T-4\tau/3}^t\int_{B_{4R/3}} (\xi \varphi_i \nabla \eta)^2 \mathrm{d}x\mathrm{d}s\\
		&\quad+4\E\int_{T-4\tau/3}^t\int_{B_{4R/3}}(\xi \eta_i \varphi_i )^2 \mathrm{d}x\mathrm{d}s +2\E\int_{T-4\tau/3}^t\int_{B_{4R/3}} (a \eta \xi \varphi)^2  \mathrm{d}x\mathrm{d}s\\ &\quad +\E\int_{T-4\tau/3}^t\int_{B_{4R/3}} (\eta \xi \varphi_{ii})^2 \mathrm{d}x\mathrm{d}s+2\E\int_{T-4\tau/3}^t\int_{B_{4R/3}} (\xi \eta b_i \varphi )^2 +(\xi \eta b \varphi_i )^2 \mathrm{d}x \mathrm{d}s.
	\end{split}
\end{equation}
This, along with  (\ref{2.61111})-(\ref{2.61114}), implies that
\begin{equation}\label{1.6}
	\begin{split}
		\displaystyle{\sup_{t\in[T-\tau,T]}}\E\int_{B_{R}}| \varphi_i (x,t)|^{2}\mathrm{d}x&\leq
		C\left(\|a\|_{\infty}^2+\|b\|_{\infty}^2\right)\E \int^{T}_{T-4\tau/3}\int_{B_{4R/3}}\varphi^2\mathrm{d}x\mathrm{d}s\\&\quad+C\left(\tau^{-1}+R^{-2}+\|b\|_{\infty}^2\right)\E \int^{T}_{T-4\tau/3}\int_{B_{4R/3}}\varphi^2_i \mathrm{d}x\mathrm{d}s\\
		&\leq C\left(\|a\|_{\infty}^2+\|b\|_{\infty}^2\right)\E \int^{T}_{T-4\tau/3}\int_{B_{4R/3}}\varphi^2\mathrm{d}x\mathrm{d}s\\&\quad+C\left(\tau^{-1}+R^{-2}+\|b\|_{\infty}^2\right)\E \int^{T}_{T-4\tau/3}\int_{B_{4R/3}}|\nabla \varphi|^2 \mathrm{d}x\mathrm{d}s.
	\end{split}
\end{equation}

According to (\ref{1.2}) of Lemma \ref{lemma-1.1} (where $r, R, \tau_{1}$ and $\tau_{2}$ are replaced by
$4R/3, 2R, 4\tau/3$ and $2\tau$, respectively), it is clear that
$$\E\int_{T-4\tau/3}^{T}\int_{B_{4R/3}}| \nabla\varphi|^{2}\mathrm{d}x\mathrm{d}t\leq C\left(\tau^{-1}+R^{-2}+\|a\|_{\infty}+\|b\|_{\infty}^2\right)\E\int_{T-2\tau}^{T}\int_{B_{2R}}\varphi^{2}\mathrm{d}x\mathrm{d}t.$$ 
This, along with (\ref{1.6}), implies that
\begin{equation*}
	\begin{split}
		\displaystyle{\sup_{t\in[T-\tau,T]}}\E\int_{B_{R}}| \varphi_i (x,t)|^{2}\mathrm{d}x&\leq
		C\left(\tau^{-2}+R^{-4}+\|a\|_{\infty}^2 +\|b\|_{\infty}^4\right)\E\int^{T}_{T-4\tau/3}\int_{B_{2R}}\varphi^2\mathrm{d}x\mathrm{d}s.
	\end{split}
\end{equation*}
Hence, (\ref{1.3}) follows from the last inequality by summing in $ i=1,...,n $.
\qed

%%%%%%%%%%%%%%%%%%%%%%%%%%%%%%%

\end{document}